\documentclass{amsart}

\usepackage{amssymb,amsmath,amscd,graphicx,amsthm,color,enumerate}

\usepackage{pst-node}
\usepackage{tikz-cd}

\newtheorem{theorem}{Theorem}
\newtheorem{lemma}[theorem]{Lemma}
\newtheorem{proposition}[theorem]{Proposition}

\theoremstyle{definition}

\theoremstyle{remark}
\newtheorem{remark}[theorem]{Remark}

\numberwithin{equation}{section}
\numberwithin{theorem}{section}

\def\A{{\mathcal A}}
\def\AA{{\mathbb A}}

\def\C{{\mathbb C}}
\def\CC{{\mathcal C}}

\def\F{{\mathbb F}}
\def\FF{{\tt F}}
\def\FFF{{\mathcal F}}
\def\G{{\mathcal G}}
\def\GCC{{\G\CC}}

\def\L{{\mathcal L}}

\def\O{{\mathcal O}}
\def\P{{\mathcal P}}

\def\UU{\overline{\A}}

\def\Y{{\mathcal Y}}
\def\Z{{\mathbb Z}}

\def\bfG{{\mathbf \Gamma}}

\def\one{\mathbf 1}

\def\phhi{{\varphi}}

\def\x{{\bf x}}

\def\Mat{\operatorname{Mat}}

\def\Poi{{\{\cdot,\cdot\}}}

\def\Van{\operatorname{Van}}

\def\deg{{\operatorname{deg}}}
\def\diag{\operatorname{diag}}

\def\:{{:\ }}

\begin{document}

\title
{Periodic Staircase Matrices and Generalized Cluster Structures}

\author{Misha Gekhtman}

\address{Department of Mathematics, University of Notre Dame, Notre Dame,
IN 46556}
\email{mgekhtma@nd.edu}

\author{Michael Shapiro}
\address{Department of Mathematics, Michigan State University, East Lansing,
MI 48823}
\email{mshapiro@math.msu.edu}

\author{Alek Vainshtein}
\address{Department of Mathematics \& Department of Computer Science, University of Haifa, Haifa,
Mount Carmel 31905, Israel}
\email{alek@cs.haifa.ac.il}
 
\begin{abstract}
As is well-known, cluster transformations in cluster structures of geometric type are often modeled on determinant identities,
such as short Pl\" ucker relations, Desnanot--Jacobi identities and their generalizations. We present a construction that plays 
a similar role in a description of generalized cluster transformations and discuss its applications to generalized cluster
structures in $GL_n$ compatible with a certain subclass of Belavin--Drinfeld Poisson--Lie brackets, in the Drinfeld double 
of $GL_n$, and in spaces of periodic difference operators.
\end{abstract}
\maketitle

\medskip

\section{Introduction}

Since the discovery of cluster algebras in \cite{FZ2}, many important algebraic varieties were shown to support a cluster structure in a sense that the coordinate rings of such variety  is  isomorphic to a cluster algebra or an upper cluster algebra. Lie theory and representation theory turned out to be a particularly rich source of varieties of this sort including but in no way limited to such examples as 
Grassmannians \cite{GSV1, Scott}, double Bruhat cells \cite{CAIII} and strata in flag varieties \cite{Leclerc}. In all these examples, cluster transformations that connect distinguished coordinate charts within a ring of regular functions are modeled on three-term relations such as short Pl\"ucker relations, Desnanot--Jacobi identities and their Lie-theoretic generalizations of the kind considered in \cite{FZ1}. This remains true even in the case of exotic cluster structures on $GL_n$ considered in \cite{GSVMem,GSVPleth} where cluster transformations can be obtained by applying Desnanot--Jacobi type identities to certain structured matrices of a size far exceeding $n$.

On the other hand, as we have shown in \cite{GSVCR,GSVDouble}, there are situations when, in order to stay within a ring of regular functions, on has to employ generalized cluster transformations, i.e. exchange relations in which the product of
a cluster variable being removed and the variable that replaces it
is equal to a multinomial expression in other cluster variables in the seed rather than a binomial expression appearing in 
the definition of the usual cluster transformation. Generalized cluster transformations of this kind were first considered in \cite{CS}, and in \cite{GSVCR,GSVDouble} we used them, in a more general form,  to construct a generalized cluster structure in the standard Drinfeld double of $GL_n$ and several related varieties. There, we had to rely on an $(n+1)$-term identity 
\cite[Proposition 8.1]{GSVDouble} (see also Proposition \ref{long_identity} below) for certain polynomial functions on the
space $\Mat_n$ of $n\times n$ matrices; this identity involved, as coefficients, conjugation invariant functions on $\Mat_n$. 

In this paper, we argue that in constructing generalized cluster structures, identities of the kind we employed in 
\cite{GSVCR,GSVDouble} play a role similar to the one classical three-term determinantal identities do in a construction of usual cluster structures. To support this argument, we derive identity \eqref{longidXY} that is associated with a class of infinite periodic block bidiagonal staircase matrices and that generalizes \cite[Proposition 8.1]{GSVDouble}. We then present three examples in which our main identity is applied to construct an initial seed of a regular generalized cluster structure.

The paper is organized as follows. In section \ref{gen_clust}, we review the definition of generalized cluster structures. Section \ref{construct} is devoted to the proof of the main identity \eqref{longidXY} (Theorem~\ref{main}).  In the next three sections we apply \eqref{longidXY} to construct generalized cluster structures on the Drinfeld double of $GL_n$ 
(Section \ref{double}), thus providing a construction alternative to the one presented in \cite{GSVCR,GSVDouble},
on the space of periodic band matrices (Section \ref{band}), and, in Section \ref{exotic}, on $GL_6$ equipped with a particular Poisson--Lie bracket arising in the Belavin--Drinfeld classification. In the latter case, the resulting generalized cluster structure is compatible with that Poisson bracket. The last section contains the proofs of several lemmas about the properties of certain minors of a periodic staircase matrices.

\section{Generalized cluster structures}
\label{gen_clust}

Following \cite{GSVDouble}, we remind the definition of a generalized cluster structure represented by a quiver 
with multiplicities. Let $(Q,d_1,\dots,d_N)$ be
a quiver on $N$ mutable and $M$ frozen vertices with positive integer multiplicities $d_i$ at mutable vertices. 
A vertex is called {\it special\/} if its multiplicity is greater than~1. A frozen vertex is called {\it isolated\/}
if it is not connected to any other vertices. Let $\F$ be the field of rational functions in $N+M$ independent variables
with rational coefficients. There are $M$  distinguished variables corresponding to frozen vertices; 
they are denoted $x_{N+1},\dots,x_{N+M}$. The {\it coefficient group\/} is a free multiplicative abelian group of Laurent monomials in stable variables, 
and its integer group ring is $\bar\AA=\Z[x_{N+1}^{\pm1},\dots,x_{N+M}^{\pm1}]$ (we write
$x^{\pm1}$ instead of $x,x^{-1}$).

An {\em extended seed\/} (of {\em geometric type\/}) in $\F$ is a triple
$\Sigma=(\x,Q,\P)$, where $\x=(x_1,\dots,x_N, x_{N+1},\dots, x_{N+M})$ is a transcendence basis of $\F$ over the field of
fractions of  $\bar\AA$ and $\P$ is a set of $N$ {\em strings}. The $i$th string is a collection of 
monomials $p_{ir}\in\AA=\Z[x_{N+1},\dots,x_{N+M}]$, $0\le r\le d_i$, such that  
$p_{i0}=p_{id_i}=1$; it is called {\em trivial\/} if $d_i=1$, and hence both elements of the string are equal to one.
The monomials $p_{ir}$ are called {\em exchange coefficients}.

Given a seed as above, the {\em adjacent cluster\/} in direction $k$, $1\le k\le N$,
is defined by $\x'=(\x\setminus\{x_k\})\cup\{x'_k\}$,
where the new cluster variable $x'_k$ is given by the {\em generalized exchange relation}
\begin{equation}\label{exchange}
x_kx'_k=\sum_{r=0}^{d_k}p_{kr}u_{k;>}^r v_{k;>}^{[r]}u_{k;<}^{d_k-r}v_{k;<}^{[d_k-r]};
\end{equation}
here $u_{k;>}$ and $u_{k;<}$, $1\le k\le N$, are 
defined by
\begin{equation*}
u_{k;>}=\prod_{k\to i\in Q} x_i,\qquad  u_{k;<}=\prod_{i\to k \in Q}x_i,
\end{equation*}
where the products are taken over all edges between $k$ and mutable vertices,
and {\em stable $\tau$-monomials\/}
$v_{k;>}^{[r]}$ and $v_{k;<}^{[r]}$, $1\le k\le N$, $0\le r\le d_k$, defined by
\begin{equation}\label{stable}
v_{k;>}^{[r]}=\prod_{N+1\le i\le N+M}x_i^{\lfloor rb_{ki}/d_k\rfloor},\qquad
v_{k;<}^{[r]}=\prod_{N+1\le i\le N+M}x_i^{\lfloor rb_{ik}/d_k\rfloor},
\end{equation}
where $b_{ki}$ is the number of edges from $k$ to $i$ and $b_{ik}$ is the number of edges from $i$ to $k$; 
here, as usual, the product over the empty set is assumed to be
equal to~$1$. 
The right hand side of~\eqref{exchange} is called a {\it generalized exchange polynomial}.

The standard definition of the {\it quiver mutation\/} in direction $k$ is modified as follows: if both vertices $i$ and $j$
in a path $i\to k\to j$ are mutable, then this path contributes $d_k$ edges $i\to j$ to the mutated quiver $Q'$; if one of the vertices $i$ or $j$ is frozen then the path contributes $d_j$ or $d_i$ edges $i\to j$ to $Q'$. The multiplicities at the vertices do not change. Note that isolated vertices remain isolated in $Q'$.

The {\em exchange coefficient mutation\/} in direction $k$ is given by
\begin{equation}
\label{CoefMutation}
 p'_{ir}=\begin{cases}
          p_{i,d_i-r}, & \text{if $i=k$;}\\
           p_{ir}, &\text{otherwise.}
        \end{cases}
\end{equation}

Given an extended seed $\Sigma=(\x,Q,\P)$, we say that a seed
$\Sigma'=(\x',Q',\P')$ is {\em adjacent\/} to $\Sigma$ (in direction
$k$) if $\x'$, $Q'$ and $\P'$ are as above. 
Two such seeds are {\em mutation equivalent\/} if they can
be connected by a sequence of pairwise adjacent seeds. 
The set of all seeds mutation equivalent to $\Sigma$ is called the {\it generalized cluster structure\/} 
(of geometric type) in $\F$ associated with $\Sigma$ and denoted by $\GCC(\Sigma)$.

Fix a ground ring $\widehat{\AA}$ such that $\AA\subseteq\widehat\AA\subseteq\bar\AA$. The
{\it generalized upper cluster algebra\/}
$\UU(\GCC)=\UU(\GCC(\Sigma))$ is the intersection of the rings of Laurent polynomials over $\widehat{\AA}$ in cluster variables taken over all seeds in $\GCC(\Sigma)$. Let $V$ be a quasi-affine variety over $\C$, $\C(V)$ be the field of rational functions on $V$, and $\O(V)$ be the ring of regular functions on $V$. A generalized cluster structure $\GCC(\Sigma)$
in $\C(V)$ is an embedding of $\x$ into $\C(V)$ that can be extended to a field isomorphism between $\F\otimes\C$ and $\C(V)$. 
It is called {\it regular on $V$\/} if any cluster variable in any cluster belongs to $\O(V)$, and {\it complete\/} if 
$\UU(\GCC)$ tensored with $\C$ is isomorphic to $\O(V)$. The choice of the ground ring is discussed in~\cite[Section 2.1]{GSVDouble}.

\begin{remark} (i) The definition above is a particular case of a more general definition of generalized
cluster structures given in \cite{GSVDouble}.

(ii) Quivers with multiplicities differ from weighted quivers introduced in \cite{LFZ}.
\end{remark}

\section{Identity for minors of a  periodic staircase matrix}
\label{construct}

Consider a periodic block bidiagonal matrix
\begin{equation}
\label{shapeL}
L= \left [
\begin{array}{ccccccc}
\ddots &\ddots &\ddots &\ddots & & &\\
 & 0 & X & Y & 0 &  &\\
& & 0 & X & Y & 0 &\\
 & & &\ddots &\ddots &\ddots &\ddots
 \end{array}
\right ],
\end{equation}
where $X\in \Mat_n$ and $Y\in GL_n$ are matrices of the form
\begin{equation}
\label{shape}
X = \left [
\begin{array}{cc}
0_{a\times b} & \ast\\ 0 & 0
\end{array}
\right ], \qquad 
Y= \left [
\begin{array}{cc}
\ast &\ast \\
0_{(n-a)\times b} &\ast
\end{array}
\right ],
\end{equation}
with $a > b + 1\ge 1$; the entries in the submatrices of $X$ and $Y$ denoted by $\ast$ can take arbitrary complex values. 
This choice ensures that $L$ has a staircase shape. Below, is an example of a dense submatrix of $L$ for 
$n = 9$, $a = 5$, $b = 2$:
{\tiny
\begingroup
\setlength{\tabcolsep}{3pt}
\begin{center}
\begin{tabular*}{\textwidth}{cccccccccccccccccccc}
$x_{17}$ & ${\underline x_{18}}$ & ${\underline x_{19}}$ & $y_{11}$ & $y_{12}$ & $y_{13}$ & $y_{14}$ & $y_{15}$ & $y_{16}$ & 
$y_{17}$ & $y_{18}$ & $y_{19}$ & & & & & & & &\\ 
$x_{27}$ & $x_{28}$ & ${\underline x_{29}}$ & ${\underline y_{21}}$ & $y_{22}$ & $y_{23}$ & $y_{24}$ & $y_{25}$ & $y_{26}$ & 
$y_{27}$ & $y_{28}$ & $y_{29}$ & & & & & & & &\\ 
$x_{37}$ & $x_{38}$ &  $x_{39}$ & ${\underline y_{31}}$ & ${\underline y_{32}}$ & $y_{33}$ & $y_{34}$ & $y_{35}$ & $y_{36}$ & 
$y_{37}$ & $y_{38}$ & $y_{39}$ & & & & & & & &\\ 
$x_{47}$ & $x_{48}$ & $x_{49}$ &   $y_{41}$ & ${\underline y_{42}}$ & ${\underline y_{43}}$ & $y_{44}$ & $y_{45}$ & $y_{46}$ & $y_{47}$ & $y_{48}$ & $y_{49}$ & & & & & & & &\\ 
$x_{57}$ & $x_{58}$ & $x_{59}$ & $y_{51}$ & $y_{52}$ & ${\underline y_{53}}$ & ${\underline y_{54}}$ & $y_{55}$ & $y_{56}$ & 
$y_{57}$ & $y_{58}$ & $y_{59}$ & & & & & & & &\\ 
           &            &   & &  & $y_{63}$ & ${\underline y_{64}}$ & ${\underline y_{65}}$ & $y_{66}$ & $y_{67}$ & $y_{68}$ & $y_{69}$ & & & & & & & &\\ 
           &            &   & &  & $y_{73}$ & $y_{74}$ & ${\underline y_{75}}$ & ${\underline y_{76}}$ & $y_{77}$ & $y_{78}$ & $y_{79}$ & & & & & & & &\\ 
           &            &   & &  & $y_{83}$ & $y_{84}$ & $y_{85}$ & ${\underline y_{86}}$ & ${\underline y_{87}}$ & $y_{88}$ & $y_{89}$ & & & & & & & &\\ 
           &            &   & &  & $y_{93}$ & $y_{94}$ & $y_{95}$ & $y_{96}$ & ${\underline y_{97}}$ & ${\underline y_{98}}$ & $y_{99}$ & & & & & & & &\\ 
           &            &   & &  & $x_{13}$ & $x_{14}$ & $x_{15}$ & $x_{16}$ & $x_{17}$ & ${\underline x_{18}}$ & 
					${\underline x_{19}}$ & $y_{11}$ & $y_{12}$ & $y_{13}$ & $y_{14}$ & $y_{15}$ & $y_{16}$ & $y_{17}$ & $y_{18}$ \\
           &  &   & &  & $x_{23}$ & $x_{24}$ & $x_{25}$ & $x_{26}$ & $x_{27}$ & $x_{28}$ & ${\underline x_{29}}$ & 
					${\underline y_{21}}$ & $y_{22}$ & $y_{23}$ & $y_{24}$ & $y_{25}$ & $y_{26}$ & $y_{27}$ & $y_{28}$ \\
           &  &   & &  & $x_{33}$ & $x_{34}$ & $x_{35}$ & $x_{36}$ & $x_{37}$ & $x_{38}$ & $x_{39}$ & ${\underline y_{31}}$ & 
					${\underline y_{32}}$ & $y_{33}$ & $y_{34}$  & $y_{35}$ & $y_{36}$ & $y_{37}$ & $y_{38}$ \\
           &            &            &             &            & $x_{43}$ & $x_{44}$ & $x_{45}$ & $x_{46}$ & $x_{47}$ & 
					$x_{48}$ & $x_{49}$ & $y_{41}$ & ${\underline y_{42}}$ & ${\underline y_{43}}$ & $y_{44}$ & $y_{45}$ & $y_{46}$ & 
					$y_{47}$ & $y_{48}$ \\
           &            &            &             &            & $x_{53}$ & $x_{54}$ & $x_{55}$ & $x_{56}$ & $x_{57}$ & 
					$x_{58}$ & $x_{59}$ & $y_{51}$ & $y_{52}$ & ${\underline y_{53}}$ & ${\underline y_{54}}$ & $y_{55}$ & $y_{56}$ & 
					$y_{57}$ & $y_{58}$ \\
           &            &            &             &            &            &            &             &            &            &            &            &             &            & $y_{63}$ & ${\underline y_{64}}$ & ${\underline y_{65}}$ & 
					$y_{66}$ & $y_{67}$ & $y_{68}$ \\
           &            &            &             &            &            &            &             &            &            &            &            &             &            & $y_{73}$ & $y_{74}$ & ${\underline y_{75}}$ & 
					${\underline y_{76}}$ & $y_{77}$ & $y_{78}$ \\
           &            &            &             &            &            &            &             &            &            &            &            &             &            & $y_{83}$ & $y_{84}$ & $y_{85}$ & ${\underline y_{86}}$ & 
					${\underline y_{87}}$ & $y_{88}$ \\
           &            &            &             &            &            &            &             &            &            &            &            &             &            & $y_{93}$ & $y_{94}$ & $y_{95}$ & $y_{96}$ & 
					${\underline y_{97}}$ & ${\underline y_{98}}$ \\
           &            &            &             &            &            &            &             &            &            &            &            &             &            & $x_{13}$ & $x_{14}$ & $x_{15}$ & $x_{16}$ & $x_{17}$ & 
					${\underline x_{18}}$ \\
           &            &            &             &            &            &            &             &            &            &            &            &             &            & $x_{23}$ & $x_{24}$ & $x_{25}$ & $x_{26}$ & $x_{27}$ & 
					$x_{28}$ \\
           &            &            &             &            &            &            &             &            &            &            &            &             &            & $x_{33}$ & $x_{34}$ & $x_{35}$ & $x_{36}$ & $x_{37}$ & 
					$x_{38}$ \\
           &            &            &             &            &            &            &             &            &            &            &            &             &            & $x_{43}$ & $x_{44}$ & $x_{45}$ & $x_{46}$ & $x_{47}$ & 
					$x_{48}$ \\
           &            &            &             &            &            &            &             &            &            &            &            &             &            & $x_{53}$ & $x_{54}$ & $x_{55}$ & $x_{56}$ & $x_{57}$ & 
					$x_{58}$ 
\end{tabular*}.
\end{center}
\endgroup
}
Denote $k = a-b$.
We say that a diagonal of $L$ is {\em inner\/} if when it is viewed as the main diagonal of $L$ then $L$ is not block-triangular. In the example above, there are two inner diagonals whose entries are underlined. In general, $L$ has 
$a - b -1 = k -1$ inner diagonals.
We define the {\em core\/} of $L$ as follows. Delete the first row in every block row of $L$, then in the resulting matrix pick the submatrix formed by consecutive rows and columns whose upper left entry is $y_{21}$ and whose lower right entry is $y_{ab}$ if $b>0$ and $x_{an}$ if $b=0$. This defines the core uniquely as a $((k-1)n+b)\times ((k-1)n+b)$ matrix
\begin{equation}
\label{Phi}
\Phi=\left [
\begin{array}{ccccc}
Y_{[2,n]} & & & & \\
X_{[2,n]} & Y_{[2,n]} & & & \\
 & \ddots &\ddots &  & \\
 & & X_{[2,n]} & Y_{[2,n]} & \\
  & & & X_{[2,a]} & Y_{[2,a]}^{[1,b]}
 \end{array}
\right ]
\end{equation}
(the $Y$-block in the lower right corner does not exist when $b=0$).
For our example above, the core is a $20\times 20$ matrix
{\tiny
\begingroup
\setlength{\tabcolsep}{3pt}
\begin{center}
\begin{tabular*}{\textwidth}{cccccccccccccccccccc}
${\underline y_{21}}$ & $y_{22}$ & $y_{23}$ & $y_{24}$ & $y_{25}$ & $y_{26}$ & $y_{27}$ & $y_{28}$ & $y_{29}$ & & & & & & & & & & & \\
${\underline y_{31}}$ & ${\underline y_{32}}$ & $y_{33}$ & $y_{34}$ & $y_{35}$ & $y_{36}$ & $y_{37}$ & $y_{38}$ & $y_{39}$ & & & & & & & & & & &  \\
$y_{41}$ & ${\underline y_{42}}$ & ${\underline y_{43}}$ & $y_{44}$ & $y_{45}$ & $y_{46}$ & $y_{47}$ & $y_{48}$ & $y_{49}$ & & & & & & & & & & & \\
$y_{51}$ & $y_{52}$ & ${\underline y_{53}}$ & ${\underline y_{54}}$ & $y_{55}$ & $y_{56}$ & $y_{57}$ & $y_{58}$ & $y_{59}$ & & & & & & & & & & & \\
            &           & $y_{63}$ & ${\underline y_{64}}$ & ${\underline y_{65}}$ & $y_{66}$ & $y_{67}$ & $y_{68}$ & $y_{69}$ & & & & & & & & & & & \\
           &            & $y_{73}$ & $y_{74}$ & ${\underline y_{75}}$ & ${\underline y_{76}}$ & $y_{77}$ & $y_{78}$ & $y_{79}$ & & & & & & & & & & & \\
           &            & $y_{83}$ & $y_{84}$ & $y_{85}$ & ${\underline y_{86}}$ & ${\underline y_{87}}$ & $y_{88}$ & $y_{89}$ & & & & & & & & & & &  \\
           &            & $y_{93}$ & $y_{94}$ & $y_{95}$ & $y_{96}$ & ${\underline y_{97}}$ & ${\underline y_{98}}$ & $y_{99}$ & & & & & & & & & & & \\
           &            & $x_{23}$ & $x_{24}$ & $x_{25}$ & $x_{26}$ & $x_{27}$ &  $x_{28}$ & ${\underline x_{29}}$ &  ${\underline y_{21}}$ & $y_{22}$ & $y_{23}$ & $y_{24}$ & $y_{25}$ & $y_{26}$ & $y_{27}$ & $y_{28}$ & $y_{29}$& & \\
           &            & $x_{33}$ & $x_{34}$ & $x_{35}$ & $x_{36}$ & $x_{37}$ & $x_{38}$ & $x_{39}$ & ${\underline y_{31}}$ & ${\underline y_{32}}$ & $y_{33}$ & $y_{34}$ & $y_{35}$ & $y_{36}$ & $y_{37}$ & $y_{38}$ & $y_{39}$ & & \\
           &            & $x_{43}$ & $x_{44}$ & $x_{45}$ & $x_{46}$ & $x_{47}$ & $x_{48}$ & $x_{49}$ & $y_{41}$ & ${\underline y_{42}}$ & ${\underline y_{43}}$ & $y_{44}$ & $y_{45}$ & $y_{46}$ & $y_{47}$ & $y_{48}$ & $y_{49}$ & & \\
           &            & $x_{53}$ & $x_{54}$ & $x_{55}$ & $x_{56}$ & $x_{57}$ & $x_{58}$ & $x_{59}$ & $y_{51}$ & $y_{52}$ & ${\underline y_{53}}$ & ${\underline y_{54}}$ & $y_{55}$ & $y_{56}$ & $y_{57}$ & $y_{58}$ & $y_{59}$& & \\
           &            &            &             &            &            &            &            &            &             &            & $y_{63}$ & ${\underline y_{64}}$ & ${\underline y_{65}}$ & $y_{66}$ & $y_{67}$ & $y_{68}$ & $y_{69}$ & & \\
           &            &            &             &            &            &            &             &            &            &            & $y_{73}$ & $y_{74}$ & ${\underline y_{75}}$ & ${\underline y_{76}}$ & $y_{77}$ & $y_{78}$ & $y_{79}$ & &  \\
           &            &            &             &            &            &            &             &            &            &            & $y_{83}$ & $y_{84}$ & $y_{85}$ & ${\underline y_{86}}$ & ${\underline y_{87}}$ & $y_{88}$ & $y_{89}$ & & \\
           &            &            &             &            &            &            &             &            &            &            & $y_{93}$ & $y_{94}$ & $y_{95}$ & $y_{96}$ & ${\underline y_{97}}$ & ${\underline y_{98}}$ & $y_{99}$ & & \\
           &            &            &             &            &            &            &             &            &            &            & $x_{23}$ & $x_{24}$ & $x_{25}$ & $x_{26}$ & $x_{27}$ & $x_{28}$ & ${\underline x_{29}}$ & ${\underline y_{21}}$ & $y_{22}$ \\
           &            &            &             &            &            &            &             &            &            &            & $x_{33}$ & $x_{34}$ & $x_{35}$ & $x_{36}$ & $x_{37}$ & $x_{38}$ & $x_{39}$ &${\underline y_{31}}$ & ${\underline y_{32}}$ \\
           &            &            &             &            &            &            &             &            &            &            &  $x_{43}$ & $x_{44}$ & $x_{45}$ & $x_{46}$ & $x_{47}$ & $x_{48}$ & $x_{49}$&$y_{41}$ & ${\underline y_{42}}$ \\
           &            &            &             &            &            &            &             &            &            &            &  $x_{53}$ & $x_{54}$ & $x_{55}$ & $x_{56}$ & $x_{57}$ & $x_{58}$ & $x_{59}$&$y_{51}$ & $y_{52}$ 
\end{tabular*}
\end{center}
\endgroup
}

Consider $n$-element segments of inner diagonals in $L$ obtained as intersections with a single block row. The main diagonal of 
$\Phi$ is made of the entries $2$ to $n$ of such segment belonging to the uppermost inner diagonal, followed by the entries $2$ to $n$ of the segment belonging to the next inner diagonal from the top, and so on, followed by entries $2$ to $n$ of the segment belonging to the lowest inner diagonal, followed by entries $x_{2,n-k+2},\ldots, x_{kn}, y_{k+1,1},\ldots ,y_{ab}$. Consequently,
each matrix entry that lies on an inner diagonal of $L$ and does not belong to the first row of $X$ or $Y$, enters the main diagonal of $\Phi$ exactly once.

For $i=1,\ldots,(k-1)n + b$ let
\begin{equation}
\label{phis}
\phhi_i =\det \Phi_{[i,(k-1)n + b]}^{[i,(k-1)n + b]}
\end{equation}
be the trailing minors of $\Phi$. In particular, $\phhi_1 =\det \Phi$ is called the {\it core determinant}.
Additionally, we set $\phhi_{(k-1)n+b+1}=1$.

We consider $\phhi_i$ as polynomials in the entries of $X$ and $Y$ indicated by $\ast$ in \eqref{shape}. 
Our goal is to establish a generalized exchange relation for $\phhi_1$ that involves the coefficients of the characteristic 
polynomial $\det(\lambda X+\mu Y)$.

Denote
\[
X Y^{-1} = \left [\begin{array}{c} W\\ 0_{(n-a)\times n} \end{array} \right ], \quad
W^{[1,a]} =  \left [\begin{array}{cc} W_{11} & W_{12}\\ W_{21} & W_{22} \end{array} \right ],\quad 
Y_{[1,a]}^{[1,b]}= \left [\begin{array}{c} Y_1\\Y_2 \end{array} \right ],
\]
where $W$ is $a\times n$, $W_{11}$ is $k\times k$, $W_{22}$ is $b\times b$, $Y_{1}$ is $k\times b$ and  $Y_{2}$ is $b\times b$. 
Let
\[
U = W_{11} - Y_1Y_2^{-1} W_{21}.
\]
If $b=0$, we set $U=W_{11}$ and use a standard convention $\det Y_2=1$.
 
\begin{lemma} For any $\lambda, \mu$,
\label{detXY}
\[
\det \left ( \lambda Y + \mu X\right ) = \lambda^{n-k} \det Y \det (\lambda \one_k + \mu U).
\]
\end{lemma}
\begin{proof} Let $t=\frac \lambda \mu$, then $\det \left ( \lambda Y + \mu X\right ) = \mu^n \det \left (tY + X\right )$. 
In turn,
\begin{equation*}
\det \left ( tY + X\right ) = \det Y \det\left ( t\one_n + \left [\begin{array} {c} W\\ 0 \end{array}\right ]\right ) 
       =t^{n-a} \det Y \det \left (  t\one_a +W^{[1,a]}_{[1,a]} \right ). 
\end{equation*}			
Note that $W^{[1,a]} Y_{[1,a]}^{[1,b]}= \left ( WY \right)^{[1,b]} = X_{[1,a]}^{[1,b]}= 0$, and so
\begin{equation}
\label{aux0}
\left [\begin{array}{cc} \one_{k} & - Y_1 Y_2^{-1} \\0 & \one_{b}\end{array} \right ] W_{[1,a]}^{[1,a]} \left [\begin{array}{cc} \one_{k} & Y_1 Y_2^{-1} \\0 & \one_{b}\end{array} \right ] =
\left [\begin{array}{cc} U & 0 \\ \star & 0\end{array} \right ];
\end{equation}
here and in what follows exact expressions for submatrices denoted by $\star$ are not relevant for further discussion.
Consequently,
\[
		\det \left (tY + X\right )=	 t^{n-a} \det Y  
       \det  \left (  t\one_a + \left [\begin{array} {cc} U & 0\\ \star & 0 \end{array}\right ]\right ) \\
        =t^{n-k} \det Y \det \left (  t\one_k +U\right ), 
\end{equation*}
and the claim follows.
\end{proof} 

As an immediate corollary from Lemma \ref{detXY}, we can write 
\begin{equation}\label{detXYcor}
\det \left ( \lambda Y + \mu X\right ) = \lambda^{n-k}  \sum_{i=0}^k c_i(X,Y) \mu^i \lambda^{k-i}, 
\end{equation}
where $c_i(X,Y)$ are polynomials in the entries of $X$ and $Y$.

\begin{theorem}\label{main}
The generalized exchange relation for the core determinant $\phhi_1$ is given by
\begin{equation}
\label{longidXY}
\phhi_1 \phhi_1^*=\sum_{i=0}^k c_i(X,Y) \left ( (-1)^{n-1}\det \bar Y\phhi_{n+1}\right )^i \phhi_2 ^{k-i},
\end{equation}
where $\phhi_1^*$ is a polynomial in the entries of $X$ and $Y$ and $\bar Y=Y_{[2,n]}^{[2,n]}$.
\end{theorem}

\begin{proof}
We start from expressing functions $\phhi_1$, $\phhi_2$ and $\phhi_{n+1}$ via $U$. 

\begin{lemma} \label{detphi1}
The core determinant $\phhi_1$ can be written as
\[
\phhi_1 = \varepsilon_1 \left ( \det Y \right )^{k-1} \det Y_2 \det \left [ U^{k-1}e_1 \ldots U^{2}e_1\  U e_1 \ e_1\right ],
\]
where $\varepsilon_1= (-1)^{n\frac{k(k-1)}{2}}$ and $e_1=(1,0,\dots,0)\in \C^k$.
\end{lemma}

\begin{lemma}\label{detphi2}
The minor $\phhi_2$ can be written as
\[
\phhi_2 = \varepsilon_2\left (\det Y\right )^{k-2} \det\bar Y\det Y_2\det
\left[ U^{k-2}v_\gamma\  U^{k-2}e_1 \ldots U^{2}e_1\  U e_1 \ e_1\right ],
\]
where  $\varepsilon_2=-\varepsilon_1$ and $v_\gamma= U (e_2 + \gamma e_1)$ with
\begin{equation}\label{gamma}
\gamma= \frac{\det Y^{[2,n]}_{1 \cup [3,n]}}{\det\bar Y}
\end{equation}
and $e_2=(0,1,0,\dots, 0)\in \C^k$. 
\end{lemma}

\begin{lemma}\label{detphilast}
The minor $\phhi_{n+1}$ can be written as
\[
\phhi_{n+1} = \varepsilon_{n+1}\left (\det Y\right )^{k-2} \det Y_2\det
\left[U^{k-2}e_1 U^{k-3}v_\gamma\  U^{k-3}e_1 \ldots U^{2}e_1\  U e_1 \ e_1\right ],
\]
where  $\varepsilon_{n+1}=(-1)^{n(k-2)(k-3)/2}$ and $v_\gamma$ is the same as in Lemma \ref{detphi2}.
\end{lemma}

Proofs of Lemmas~\ref{detphi1}--\ref{detphilast} are given in Section~\ref{detlemmas}.

Now we can invoke a result proven in \cite[Proposition 8.1]{GSVDouble}.

\begin{proposition}
\label{long_identity}
Let $A$ be a complex $k\times k$ matrix. For  $u,v\in \C^k$, define matrices
\begin{gather*}\nonumber
 K(A;u)=\left[ u \; A u \; A^2 u \dots  A^{k-1} u\right],\\ 
 K_1(A;u,v)=\left[ v \; u  \; A u \dots A^{k-2} u\right],\;\; 
 K_2(A;u,v)=\left[A v \; u  \;A u  \dots  A^{k-2} u\right].
\end{gather*}
In addition, let $w$ be the last row of the classical adjoint of $ K_1(A;u,v)$, so that
$w  K_1(A;u,v) = \left (\det  K_1(A;u,v) \right )e_k^T$. Define
$ K^*(A;u,v)$ to be the matrix with rows $w, w A, \ldots, w A^{k-1}$. Then
\begin{equation}
\label{longid}
\det\Big(\det K_1(A;u,v) A -  \det K_2(A;u,v) \one_k \Big ) = (-1)^{\frac{k(k-1)}{2}} \det K(A;u) \det K^*(A;u,v).
\end{equation}
\end{proposition}

We will make use of the following properties of matrices $K$ and $K^*$.

\begin{lemma}\label{gencop}
{\rm (i)} For any $\gamma\in\C$ there exists an invertible matrix $A$ such that
$\det K(A; e_1)=0$, but $\det K^*(A; e_1, A^{-1}(e_2+\gamma e_1))\ne 0$.

{\rm (ii)} Moreover, $A$ can be chosen in such a way that all principal leading minors of $A$ do not vanish.
\end{lemma}

The proof of the Lemma is given in Section~\ref{varcop}.

Using notation introduced in Proposition \ref{long_identity}, we can re-write the claims of 
Lemmas \ref{detphi1}--\ref{detphilast} as
\begin{align*}
\det  K(U^{-1};e_1)&= \varepsilon_1 \left ( \det Y \right )^{1-k}   \left (\det Y_2 \right )^{-1}  \left ( \det U \right )^{1-k} \phhi_1,\\
\det  K_1(U^{-1};e_1,v_\gamma)&= \varepsilon_2 \left ( \det Y \right )^{2-k}  \left ( \det\bar Y\right )^{-1}  
\left (\det Y_2 \right )^{-1}  \left ( \det U \right )^{2-k} \phhi_2,\\
\det  K_2(U^{-1};e_1,v_\gamma)&= -\varepsilon_{n+1} \left ( \det Y \right )^{2-k}   \left (\det Y_2 \right )^{-1}  
\left ( \det U \right )^{2-k} \phhi_{n+1}.
\end{align*}
Consequently, the matrix in the left hand side of \eqref{longid} equals
\[
\left ( \det Y \right )^{2-k}  \left ( \det\bar Y\right )^{-1}  \left (\det Y_2 \right )^{-1} \left ( \det U \right )^{2-k} 
\left (\varepsilon_2\phhi_2 \one_k +  \varepsilon_{n+1} \det\bar Y   \phhi_{n+1} U\right )U^{-1},
\]
and \eqref{longid}  becomes
\begin{equation}
\label{longidU}
\begin{split}
\det &\left ( \varepsilon_2\phhi_2 \one_k +  \varepsilon_{n+1} \det\bar Y   \phhi_{n+1} U\right ) \\
       & = (-1)^{\frac{k(k-1)}{2}} \varepsilon_1 
			\phhi_1 \frac {\det K^*(U^{-1};e_1,v)}{\det Y} c_k(X,Y)^{(k-1)(k-2)} 
			\left (\det Y_2 \right )^{k-1}  \left ( \det\bar Y\right )^{k},
\end{split}
\end{equation}
since $c_k(X,Y)=\det Y\det U$.

Using Lemma \ref{detXY} and equations \eqref{detXYcor}, \eqref{longidU} we get \eqref{longidXY} with 
\begin{equation}
\label{phi_star}
 \phhi_1^* = (-1)^{\frac{k(k-1)}{2}} \varepsilon_1\varepsilon_2^k\det K^*(U^{-1}; e_1,v_\gamma) c_k(X,Y)^{(k-1)(k-2)}  
\left ( \det Y_2\right )^{k-1}\left ( \det \bar Y\right )^k. 
\end{equation}

Note that 
$\det K^*(U^{-1}; e_1,v_\gamma)$  is a rational function of $X, Y$ whose denominator can contain only powers of $\det Y$,  
$\det Y_2$ and $\det \bar Y$. It remains to establish that $\phhi_1^*$ is a polynomial function of $X$ and $Y$. 
By \eqref{longidXY}, this fact is an immediate corollary of the following statement.

\begin{lemma}\label{coreirr} 
The core determinant $\phhi_1=\phhi_1(X,Y)$ is an irreducible polynomial in the entries of $X$ and $Y$.
\end{lemma}

The proof of the Lemma is given in Section~\ref{varirr}.
\end{proof}

\section{Example 1: A generalized  cluster structure on the Drinfeld double of $GL_n$}
\label{double}

In \cite{GSVCR, GSVDouble} we presented a generalized cluster structure  on the standard Drinfeld double 
$D(GL_n)=GL_n\times GL_n$ and studied its properties. In this section, we explain how the construction of 
Section \ref{construct} can be applied to select a different initial seed for a generalized cluster structure on $D(GL_n)$.

In this case $X$ and $Y$ in \eqref{shapeL} are arbitrary $n\times n$ matrices, and hence $b=0$ and $a=k=n$. Consequently, the core $\Phi$
is an $N\times N$  matrix 
\[
\Phi=\Phi\left ( X, Y\right)=\left (
\begin{array}{cccc}
Y_{[2,n]} & & &  \\
X_{[2,n]} & Y_{[2,n]} & &  \\
 & \ddots &\ddots &   \\
 & & X_{[2,n]} & Y_{[2,n]}  \\
  & & & X_{[2,n]} 
 \end{array}
\right )
\]
with $N=(n-1)n$ and $\phhi_i=\det\Phi_{[i,N]}^{[i,N]}$. 
Further, we have $U=W=X Y^{-1}$ and
$\det \left ( \lambda Y + \mu X\right ) = \sum_{i=0}^n c_i(X,Y) \mu^i \lambda^{n-i}$.

Following \cite{GSVDouble}, we define $g_{ij}=\det X_{[i,n]}^{[j,j+n-i]}$ for $1\le j\le i\le n$, 
and, $h_{ij}=\det Y_{[i,i+n-j]}^{[j,n]}$ for $1\le i\le j\le n$; note that $\phhi_i=g_{i-N+n-1,i-N+n-1}$ for $i>N-n+1$,
and that $h_{22}=\bar Y$.
The family ${\FFF_n}$ of $2n^2$ functions in the ring of regular functions on $D(GL_n)$ is defined as
\[
\FFF_n=\left\{\{\phhi_i\}_{i=1}^{N-n+1};\ \{g_{ij}\}_{1\le j\le i\le n};\ \{h_{ij}\}_{1\le i\le j\le n};\ 
\{\tilde c_i\}_{i=1}^{n-1}\right\}
\]
with $\tilde c_i(X,Y)=(-1)^{i(n-1)}c_i(X,Y)$ for $1\le i\le n-1$.

The corresponding quiver $Q_n$ is defined below and illustrated, for the $n=4$ case, in Figure \ref{example_quiver}. 
It has $2n^2$ vertices corresponding to the functions in $\FFF_n$. The $n-1$ vertices corresponding to 
$\tilde c_i(X,Y)$, $1\le i\le n-1$, are isolated; they are not shown. There are $2n$ frozen vertices corresponding to 
$g_{i1}$, $1\le i\le n$, and $h_{1j}$, $1\le j\le n$;  they are shown as squares in the figure below.
All vertices except for one are arranged into a $(2n-1) \times n$ grid; we will 
refer to vertices of the grid using their position in the grid numbered top to bottom and left to right.
The edges of $Q_n$ are $(i,j) \to (i+1,j+1)$ for $i=1,\dots, 2n-2$, $j=1,\ldots, n-1$, $(i,j) \to (i,j-1)$ and 
$(i,j) \to (i-1,j)$
for $i=2,\dots, 2n-1$, $j=2,\ldots, n$, and $(i,1)\to (i-1,1)$ for $i=2,\dots,n$. Additionally, there is an oriented path
\[
(n+1,n)\to (3,1)\to (n+2,n)\to (4,1)\to\cdots(n,1)\to(2n-1,n). 
\]
The edges in this path are depicted as dashed in Figure \ref{example_quiver}. The vertex $(2,1)$ is special;
it is shown as a hexagon in the figure. The last remaining 
vertex of $Q_n$ is placed to the left of the special vertex and there is an edge pointing from the former one to the latter.

Functions $h_{ij}$ are attached to the vertices $(i,j)$, $1\le i\le j\le n$, and all vertices in the upper row of $Q_n$
are frozen. Functions $g_{ij}$ are attached to the vertices $(n+i-1,j)$, $1\le j\le i\le n$, $(i,j)\ne(1,1)$, and all such
vertices in the first column are frozen. The function $g_{11}$ is attached to the vertex to the left of the special one, and
this vertex is frozen. Functions $\phhi_{kn+i}$ are attached to the vertices $(i+k+1,i)$ for $1\le i\le n$, $0\le k\le n-3$;
the function $\phhi_{N-n+1}$ is attached to the vertex $(n,1)$. All these vertices are mutable. The set of strings $\P_{n}$ contains a unique nontrivial string $(1,\tilde c_1(X,Y),\dots,\tilde c_{n-1}(X,Y),1)$ corresponding to the unique 
special vertex.

\begin{figure}[ht]
\begin{center}
\includegraphics[width=8cm]{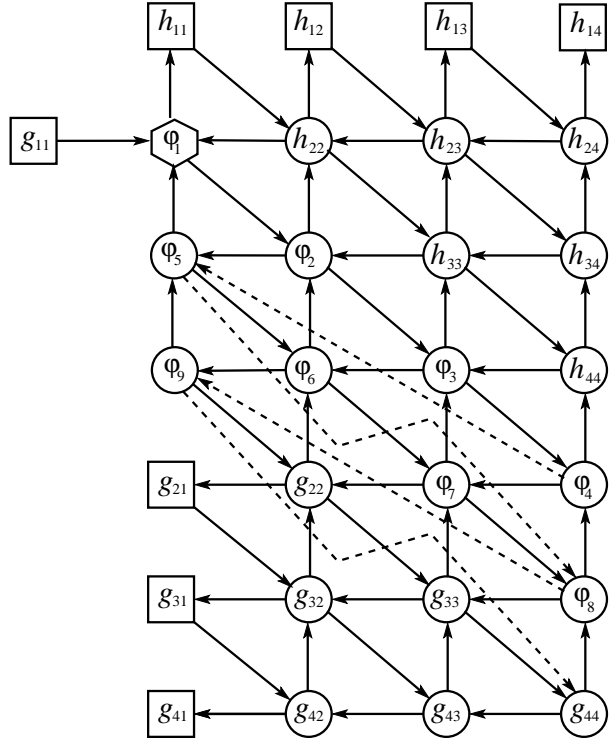}
\end{center}
\caption{Quiver $Q_4$}
\label{example_quiver}
\end{figure}

\begin{theorem}
\label{structure}
The extended seed $\Sigma_n=(\FFF_n,Q_n,\P_n)$ defines 
a regular generalized cluster structure on $D(GL_n)$.
\end{theorem}

\begin{proof}
We start with checking that relation~\eqref{longidXY} with $k=n$ 
indeed defines a generalized exchange relation as described in~\eqref{exchange}. The degree of the exchange relation is 
$d_n=n$, exchange coefficients are given by $p_{1r}=\tilde c_r(X,Y)$ for 
$r=1,\dots,n-1$, the cluster $\tau$-monomials are $u_{1;>}=h_{22}\phhi_{n+1}$ and $u_{1;<}=\phhi_{2}$. 
The stable $\tau$-monomials are defined as follows:
\begin{equation*}
v_{1;>}^{[n]}=h_{11}=\det Y,\qquad
v_{1;>}^{[r]}=1\quad\text{for $0\le r\le n-1$},
\end{equation*}		
and
\begin{equation*}
v_{1;<}^{[n]}=g_{11}=\det X,\qquad
v_{1;<}^{[r]}=1\quad\text{for $0\le r\le n-1$}.
\end{equation*}					

Let us show that cluster transformations defined by the quiver $Q_{n}$ produce regular functions. For the special vertex this
follows from Theorem \ref{main}. For the vertices corresponding to $g_{ij}$ and $h_{ij}$ with $i\ne j$ the claim is well-known
from the study of the standard cluster structure on $GL_n$. 
For other mutable vertices  we use determinantal identities
 often utilized for this purpose
(see, e.g., \cite{GSVMem}, \cite{GSVDouble}, \cite{GSVPleth}).
The first is the Desnanot--Jacobi identity for minors of a square matrix $A$:
\begin{equation}
\label{jacobi}
\det A \det A_{\hat \alpha \hat \beta}^{\hat \gamma \hat \delta} + \det A_{\hat \alpha }^{ \hat \delta} \det A_{\hat \beta}^{ \hat \gamma}= 
\det A_{\hat \alpha }^{ \hat \gamma} \det A_{\hat \beta}^{ \hat \delta} ,
\end{equation}
where ``hatted'' subscripts and superscripts indicate deleted rows and columns, respectively. 
The second is a version of a short Pl\"ucker relation for an $m\times (m+1)$ matrix $B$:
\begin{equation}
\label{pluck}
\det B^{\hat\alpha\hat\beta}_{\hat\delta} \det B^{\hat\gamma} +
\det B^{\hat\beta\hat\gamma}_{\hat\delta} \det B^{\hat\alpha} = 
\det B^{\hat\alpha\hat\gamma}_{\hat\delta} \det B^{\hat\beta},
\end{equation}
and the third is the corollary of \eqref{pluck}:
\begin{equation}
\label{pluckpluck}
\begin{split}
&\det B_{\hat 1\hat 2}^{\hat 1\widehat m\widehat{m+1}}\det B^{\hat 1\hat 2}_{\hat 1} \det B^{\widehat{m+1}} +
\det B_{\hat 1\hat 2}^{\hat 1\hat 2\widehat{m+1}}\det B^{\widehat m\widehat{m+1}}_{\hat 1} \det B^{\hat{1}} \\
& \qquad = \det B^{\hat 1\widehat{m+1}}_{\hat 1} 
\left (
\det B_{\hat 1\hat 2}^{\hat 1\widehat m\widehat{m+1}} \det B^{\widehat{2}} - 
\det B_{\hat 1\hat 2}^{\hat 2\widehat m\widehat{m+1}}  \det B^{\hat{1}} 
\right ).
\end{split}
\end{equation}

In more detail, for functions $\phhi_i$ with $2\le i\le n-1$ we use~\eqref{pluckpluck} for the matrix 
$B=\left[ \Phi \;\; e_N^T\right]_{[i-1,N]}^{[i-1,N+1]}$. For $\phhi_n$ we use~\eqref{jacobi} for the matrix 
$A=\Phi_{[n-1,N]}^{[n-1,N]}$ with parameters $\alpha=\gamma=1$, $\beta=2$, $\delta=N-n+2$. 
For functions $\phhi_i$ with
$n+1\le i\le N-1$ we consider a perturbation $\Phi(\theta)=\Phi+\theta e_{(n-1)^2+1,(n-1)^2-1}$ of the core and
use~\eqref{pluckpluck} for the matrix $B(\theta)=\Phi(\theta)_{[i-n,N]}^{[i-n-1,N]}$ (for $i=n+1$ the range of columns 
$[0,N]$ stands for $\Phi$ prepended with the previous column of the infinite periodic matrix~\eqref{shapeL}; this column contains $X_{[2,n]}^{[n]}$ to the left of the uppermost copy of $Y_{[2,n]}$ in~\eqref{Phi}). A direct check shows that the 
identity~\eqref{pluckpluck} for $B(\theta)$ yields a polynomial
of degree~3 in $\theta$ that vanishes identically. The coefficient of this polynomial at $\theta$ is the exchange relation
we are looking for. For $\phhi_N$ we use~\eqref{pluck} for the matrix $B=\Phi_{[N-n,N]}^{[N-n-1,N]}$ with parameters
$\alpha=\delta=1$, $\beta=2$, $\gamma=n+2$. For functions $h_{ii}$ with $3\le i\le n$ we consider a perturbation
$\overline\Phi(\theta)=\left[ \Phi \;\; e_N^T\;\; e_N^T\right]+\theta e_{n-1,n+1}+\theta e_{N-1,N+1}$ and
use~\eqref{pluckpluck} for the matrix $\overline{B}(\theta)=\overline\Phi(\theta)_{[i-2,N]}^{[i-1,N+2]}$. 
A direct check shows that the identity~\eqref{pluckpluck} for $\overline{B}(\theta)$ yields a polynomial
of degree~4 in $\theta$ that vanishes identically. The coefficient of this polynomial at $\theta^2$ is the exchange relation
we are looking for. Finally, for $h_{22}$ we prepend a row $[Y_{[1]}\; 0]$ to the matrix $\overline\Phi(\theta)$ and proceed
with the obtained matrix exactly as in the previous case. 

By \cite[Proposition 2.3]{GSVDouble}, it remains to check that any two functions in $\FFF_n$ are coprime and 
that for any non-frozen $f\in\FFF_n$, 
the function $f^*$ that replaces $f$ after the mutation is coprime with $f$. 
The first claim above is an immediate corollary of the following statement.

\begin{lemma}\label{fnirr}
All functions in the family $\FFF_{n}$ are irreducible.
\end{lemma}

The proof of Lemma~\ref{fnirr} is given in Section~\ref{varirr}. The second claim above is provided by the following statement. 

\begin{lemma}\label{fncop} 
Every non-frozen $f\in\FFF_n$ does not divide the corresponding $f^*$.
\end{lemma}  

 The proof of Lemma~\ref{fncop} is given in Section~\ref{varcop}. 
\end{proof}

\begin{remark} (i) We expect that the regular generalized cluster structure described in Theorem \ref{structure}
is complete in $\O(D(GL_{n}))$ and compatible with the standard Poisson--Lie bracket on $D(GL_n)$.

(ii) In \cite{GSVDouble} we used a different initial seed $\widetilde\Sigma_n$ to define a regular complete 
generalized cluster structure $\GCC(\widetilde\Sigma_n)$ on $D(GL_n)$ compatible with the standard Poisson--Lie 
structure on $D(GL_n)$. Moreover, the sets of frozen variables for both structures coincide. However, 
there is an evidence suggesting that the initial seed
described above is not mutation equivalent to the one constructed
in \cite{GSVDouble}.
\end{remark}

Details and proofs of assertions mentioned in the above remark will be considered in a separate publication.

\section{Example 2: Generalized cluster structure on periodic band matrices.}
\label{band}

In this section we consider the case of $L$ in \eqref{shapeL} being a $(k+1)$ diagonal $n$-periodic band matrix with $k< n$.
In other words, $L$ represents a periodic difference operator. Such operators play an important role in spectral theory; they also appear as Lax operators in the theory of integrable systems, such as periodic Toda lattices and their multicomponent analogues (see, e.g. \cite{vMM}).
More recently,  periodic difference operators found applications that, in turn, proved to be related to the 
theory of  cluster algebras, in particular, in the investigation  of frieze patterns and pentagram maps and their generalizations \cite{MOST, Anton}. In this section, we will use Theorem \ref{main} to construct a generalized cluster algebra structure on the space of periodic difference operators.

We choose $Y$ in \eqref{shape} to be a lower triangular band matrix with $k+1$ non-zero diagonals 
(including the main diagonal); consequently, $X$ is an upper 
triangular with zeroes everywhere outside of $(k-1)\times (k-1)$ upper triangular block in the upper right corner. 
We assume that entries of the lowest and highest diagonals are all nonzero.
$X$ and $Y$ are now $n\times n$ matrices of the form
\begin{equation}
\label{XYband}
\begin{aligned}
X &= \left [
\begin{array}{cccccc}
0 & \cdots & 0 & a_{11} & \cdots & a_{k1}\\
0 & \cdots& 0 & 0 & a_{12} & \cdots \\
\vdots & \vdots & \vdots & \vdots & \ddots& \ddots \\
0 & \cdots & 0 & \cdots  & 0 & a_{1k}\\
0 & \cdots & 0 & \cdots & \cdots & 0\\
\vdots & \vdots & \vdots & \vdots & \vdots& \vdots \\
\end{array}
\right ],\\
Y&=\left [
\begin{array}{cccccc}
a_{k+1,1} & 0 & \cdots & \cdots& \cdots & 0\\
a_{k2} & a_{k+1,2}  & 0 & \cdots & \cdots & \cdots \\
\vdots & \vdots & \ddots & \vdots & \vdots& \vdots \\
a_{1,k+1}& a_{2,k+1}& \cdots  & a_{k+1,k+1} & 0 &\cdots\\
0  & \ddots & \ddots & \vdots & \ddots& \vdots \\
0 & \cdots & a_{1n} & a_{2n} &\cdots & a_{k+1,n}
\end{array}
\right ],
\end{aligned}
\end{equation}
and we can choose $a=k$, $b=0$. Consequently, 
\begin{equation}\label{bandU}
U= W_{11} = \left [ \begin{array}{ccc} a_{11} &\cdots & a_{k1}\\
\vdots &  \ddots & \vdots \\
0 & \cdots & a_{1k}\end{array}  \right ] \left (Y^{-1}  \right )_{[n-k+1,n]}^{[1,k]},
\end{equation}
and hence 
\begin{equation}\label{banddetU}
\det U = \frac{a_{11}\cdots a_{1n}} {a_{k+1,1}\cdots a_{k+1,n}},\qquad
c_k(X,Y) = a_{11}\cdots a_{1n}.
\end{equation}
Furthermore, $\gamma$ in \eqref{gamma} is equal to $0$, and therefore $v_\gamma= U e_2$.

The core $\Phi$ is a reducible $(k-1)n\times(k-1)n$ matrix, 
and for $i=1,\ldots, (k-1)(n-1)$ we have  $\phhi_i = \tilde\phhi_i a_{12}\cdots a_{1k}$ with 
\begin{equation}
\label{tildephi}
\tilde\varphi_i = \det \Phi_{[i,(k-1)(n-1)]}^{[i,(k-1)(n-1)]}.
\end{equation}

Relation \eqref{longidXY} can be rewritten as
\begin{equation}
\label{longidXYtri} 
\tilde\phhi_1 \phhi_1^*=
\left (a_{12}\cdots a_{1k} \right )^{k-1}\sum_{i=0}^k c_i(X,Y) \left ( (-1)^{n-1} \det \bar Y  \tilde\phhi_{n+1}\right )^i \tilde\phhi_2^{k-i},
\end{equation}
for $k>2$ and as 
\begin{equation}
\label{longidXYtri2}
\tilde\phhi_1 \phhi_1^*=
c_0(X,Y) \tilde\phhi_2^{2}a_{12}+(-1)^{n-1} c_1(X,Y)  \det \bar Y \tilde\phhi_2 +  
c_2(X,Y) \left ( \det \bar Y\right )^2a_{12}^{-1},
\end{equation}
for $k=2$, since in this case $\phhi_{n+1}=\phhi_{(k-1)n+b+1}=1$ according to the convention introduced in 
Section \ref{construct}. 
 In both cases,  $\phhi_1^*$ is a polynomial function in matrix entries of $X, Y$, according to Theorem \ref{main}. Since 
$c_0(X,Y)=\det Y =a_{k+1,1} \det\bar Y $, the right hand side of \eqref{longidXYtri} is divisible by 
$\det \bar Y=a_{k+1,2}\cdots a_{k+1,n}$. On the other hand, it is easy to see that $\tilde\phhi_1$ is not divisible by 
$a_{1 i}$, $a_{k+1,i}$ for $i=2,\ldots, n$. This means that for $k > 2$,
\[
 \phhi_1^* = \left (a_{12}\cdots a_{1k} \right )^{k-1} \det \bar Y \tilde \varphi_1^*,
\]
where $\tilde\varphi_1^*$ is a polynomial function in matrix entries of $X$ and $Y$. Thus, \eqref{longidXYtri} becomes
\begin{equation}
\label{longidXYtritilde}
\tilde\phhi_1\tilde\phhi_1^*=a_{k+1,1}\tilde\phhi_2 ^{k}+
\sum_{i=1}^k \tilde c_i(X,Y) (\det \bar Y)^{i-1} \tilde\phhi_{n+1}^i \tilde\phhi_2 ^{k-i},
\end{equation}
where  $\tilde c_i(X,Y)=(-1)^{i(n-1)}c_i(X,Y)$ for $1\le i\le k$. In what follows, it will be convenient to introduce
$\tilde a_{11}=(-1)^{k(n-1)}a_{11}$, so that $\tilde c_k(X,Y)=\tilde a_{11}a_{12}\cdots a_{1n}$.

Similarly, for $k=2$, $\phhi_1^*=\tilde\phhi_1^*\det \bar Y$ where $\tilde\phhi_1^*$ is a polynomial function
in matrix entries of $X$ and $Y$, and \eqref{longidXYtri2} becomes
\begin{equation}
\label{longidXYtri21}
\tilde\phhi_1 \tilde\phhi_1^*=
a_{31} a_{12} \tilde\phhi_2^2 + \tilde c_1(X,Y)\tilde\phhi_2 + \tilde c_2(X,Y) \det \bar Y
\end{equation}
with $\tilde c_1(X,Y)= (-1)^{n-1}c_1(X,Y)$ and  
$\tilde c_2(X,Y)= c_2(X,Y)/a_{12} = a_{11}a_{13}\cdots a_{1n}$.

For $k< n$, denote by $\L_{kn}$ the space of periodic difference operators represented by
$n$-periodic $(k+1)$-diagonal matrices with all entries of the lowest and the highest diagonals nonzero.
A generalized cluster structure in the space of regular functions on $\L_{kn}$ is defined by the following data.

Consider the family $\FFF_{kn}$ of functions  on $\L_{kn}$: 
\[
\FFF_{kn}=\left\{\{\tilde\phhi_{i}\}_{i=1}^{(k-1)(n-1)};\ \tilde a_{11};\ \{a_{1i}\}_{i=2}^n;\ \{a_{k+1,i}\}_{i=1}^n;\
\{\tilde c_i(X,Y)\}_{i=1}^{k-1}\right\}.
\]

Let $Q_{kn}$ be the quiver with $(k+1)n$ vertices, of which $k-1$ vertices are isolated and are not shown in the figure below, 
$(k+1)(n-1)$ are arranged in an $(n-1)\times (k+1)$ grid and denoted
$(i,j)$, $1\le i\le n-1$, $1\le j\le k+1$, and the remaining two are placed on top of the leftmost and the rightmost 
columns in the grid and denoted $(0,1)$ and $(0,k+1)$, respectively.  All vertices in the leftmost and in the 
rightmost columns are frozen. The vertex $(1,k)$ is special, and its multiplicity equals $k$. All other vertices are regular mutable vertices.

The edge set of $Q_{kn}$
consists of the edges $(i,j)\to (i+1,j)$ for $i=1,\ldots, n-2$, $j= 2,\ldots, k$; 
$(i,j)\to (i,j-1)$ for $i=1,\ldots, n-1$, $j= 2,\ldots, k$, $(i,j)\ne (1,k)$; 
$(i+1,j)\to (i,j+1)$ for $i=1,\ldots, n-2$, $j= 2,\ldots, k$, shown by solid lines. 
In addition, there are edges  $(n-1,3)\to (1,2)$, $(1,2)\to (n-1,4)$,
$(n-1,4)\to (1,3),\ldots, (1,k-1)\to(n-1,k+1)$ that form a directed path (shown by dotted lines). 
Save for this path, and the missing edge 
$(1,k)\to (1,k-1)$, mutable vertices of $Q_{kn}$ form a mesh of consistently oriented triangles

Finally, there are edges between the special vertex $(1, k)$ and frozen vertices $(i,1)$, $(i,k+1)$ for $i=0,\ldots n-1$. 
There are $k-1$ parallel edges between $(1,k)$ and $(i,k+1)$ for $i=1,\dots,n-1$, 
and one edge between $(1,k)$ and all other frozen vertices 
(including $(0,k+1)$). If $k>2$, all of these edges are directed towards $(1, k)$, and if $k=2$, the direction
of the edge between $(1,1)$ and $(1,k)$ is reversed. Quiver $Q_{47}$ is shown in Figure~\ref{Qkn}. 

\begin{figure}[ht]
\begin{center}
\includegraphics[width=10cm]{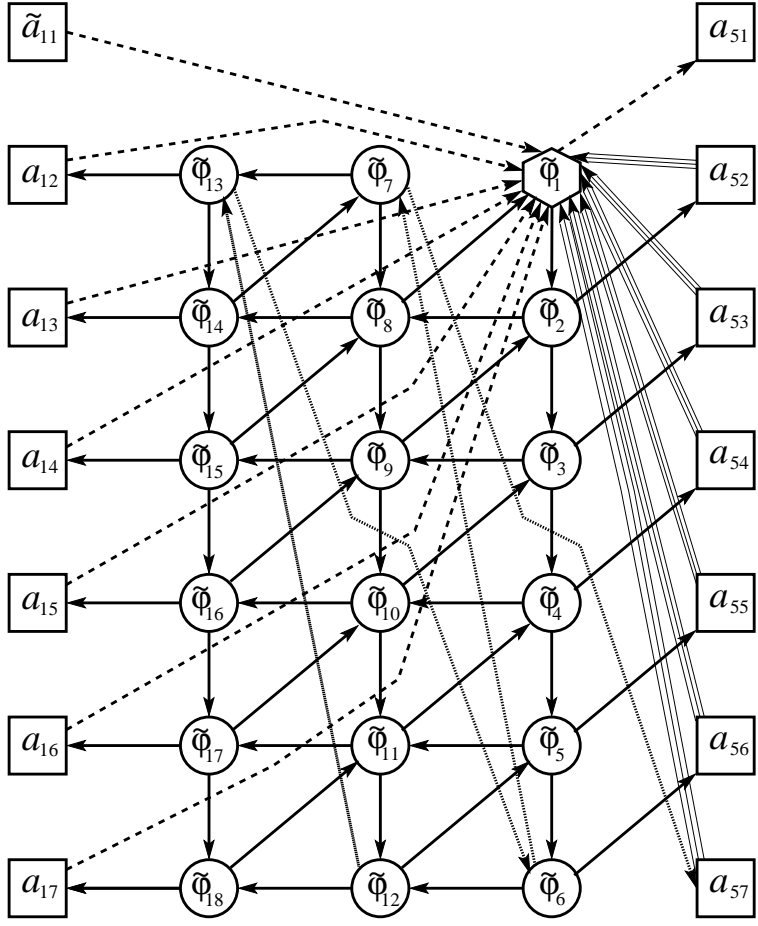}
\end{center}
\caption{Quiver $Q_{47}$}
\label{Qkn}
\end{figure}

We attach functions $\tilde a_{11}, a_{12}, \ldots, a_{1n}$, in a top to bottom order, to the vertices of the leftmost 
column in $Q_{kn}$, and functions $a_{k+1,1},\dots, a_{k+1,n}$, in the same order, to the vertices of the rightmost column
in $Q_{kn}$. Functions $\tilde\phhi_{i}$ are attached, in a 
top to bottom, right to left order, to the remaining vertices of $Q_{kn}$, starting with $\tilde\phhi_{1}$ attached to the special vertex $(1,k)$. The set of strings $\P_{kn}$ contains a unique nontrivial string 
$(1,\tilde c_1(X,Y),\dots,\tilde c_{k-1}(X,Y),1)$ corresponding to the unique special vertex.

\begin{theorem}
\label{Band_structure}
The extended seed $\Sigma_{kn}=(\FFF_{kn}, Q_{kn},\P_{kn)}$ defines a regular generalized cluster structure  
$\GCC(\Sigma_{kn})$ on $\L_{kn}$. 
\end{theorem}

\begin{proof} Similarly to the proof of Theorem~\ref{structure}, let us check first that
relation \eqref{longidXYtritilde} indeed defines a generalized cluster transformation as described in \eqref{exchange}.
The degree of the exchange relation is $d_k=k$, exchange coefficients are given by $p_{1r}=\tilde c_r(X,Y)$ for 
$r=1,\dots,k-1$, the cluster $\tau$-monomials are $u_{1;>}=\tilde\phhi_2$ and $u_{1;<}=\tilde\phhi_{n+1}$ for $k>2$
(for $k=2$,  $u_{1;<}=1$). The stable $\tau$-monomials are defined as follows:
\begin{equation*}
v_{1;>}^{[k]}=\begin{cases} a_{k+1,1}\quad\text{if $k>2$},\\
                            a_{31}a_{12}\quad\text{if $k=2$},\end{cases}\qquad
v_{1;>}^{[r]}=1\quad\text{for $0\le r\le k-1$},
\end{equation*}		
and
\begin{equation*}
\begin{aligned}
v_{1;<}^{[k]}&=\begin{cases} \tilde a_{11}a_{12}\dots a_{1n}a_{k+1,2}^{k-1}\dots a_{k+1,n}^{k-1}\quad\text{if $k>2$},\\
                            a_{11}a_{13}\dots a_{1n}a_{k+1,2}\dots a_{k+1,n} \quad\text{if $k=2$},\end{cases}\\
v_{1;<}^{[r]}&=a_{k+1,2}^{r-1}\dots a_{k+1,n}^{r-1}\quad\text{for $1\le r\le k-1$},\\
v_{1;<}^{[0]}&=1;
\end{aligned}
\end{equation*}					
expression for $v_{1;<}^{[r]}$ follows from \eqref{stable} via $\lfloor(k-1)r/k\rfloor=r-1$.

Let us show that cluster transformations defined by the quiver $Q_{kn}$ produce regular functions. For the special vertex this
follows from Theorem \ref{main}. For other mutable vertices  we use determinantal identities~\eqref{jacobi}--\eqref{pluckpluck}.

In more detail, consider a perturbation 
\[
\Phi(\theta)=\Phi+\theta\sum_{i=1}^{n-1}e_{(k-2)(n-1)+i,(k-2)(n-1)+i-2}
\] 
of the core.  For every six-valent vertex $(i,j)$  in $Q_{kn}$
we apply~\eqref{pluckpluck} to the submatrix 
$B(\theta)=\Phi(\theta)_{[(k-j-1)(n-1) + i-1,(k-1)(n-1)]}^{[(k-j-1)(n-1) + i-2,(k-1)(n-1)]}$ 
of $\Phi(\theta)$ and get a polynomial  identity of degree $3(n-1)$ in $\theta$. 
The claim follows from considering the coefficient at $\theta^{n-1}$. 
Indeed, the submatrix $\Phi_{[t,(k-1)(n-1)]}^{[t,(k-1)(n-1)]}$ that defines the function $\tilde\phhi_t$
coincides with the submatrix of $\Phi$ of the same size with the upper left corner at row $t-s(n-1)$ and column
$t-sn$ for $s=1,2,\dots$. Note that for the function attached to $(i,j)$ we have $t=(k-j)(n-1)+i$, and the result follows. 

 For vertices $(i, 2)$, $i=1,\ldots, n-1$, one needs to apply \eqref{pluck} 
to the submatrix $B=\Phi_{[(k-3)(n-1) + i-1,(k-1)(n-1)]}^{[(k-3)(n-1) + i-2,(k-1)(n-1)]}$ with $\alpha=\delta=1$, $\beta=2$,
and $\gamma$ being the last row. The same holds for the vertex $(n-1,3)$ with $i=0$.
Finally, for vertices $(i,k)$, $i=2,\dots,n-1$, one needs to apply \eqref{jacobi} to
the submatrix $A=\Phi_{[i-1,(k-1)(n-1)]}^{[i-1,(k-1)(n-1)]}$ with $\alpha=\gamma=1$, $\beta=2$, and $\delta$ being the last column. The vertex $(1,k-1)$ is treated in the same way.

Similarly to the proof of Theorem~\ref{structure}, it remains to prove that all functions in $\FFF_{kn}$ are coprime, 
and that each non-frozen $f\in\FFF_{kn}$ is coprime with $f^*$. The first of the above claims is an immediate corollary
of the following statement.

\begin{lemma}\label{fknirr}
All functions in the family $\FFF_{kn}$ are irreducible.
\end{lemma}

The proof of Lemma~\ref{fknirr} is given in Section~\ref{varirr}.  The second claim above is provided by the following statement. 

\begin{lemma}\label{fkncop} 
Every non-frozen $f\in\FFF_{kn}$ does not divide the corresponding $f^*$.
\end{lemma}  

 The proof of Lemma~\ref{fkncop} is given in Section~\ref{varcop}. 
\end{proof}

\begin{remark} (i) We expect that the regular generalized cluster structure described in Theorem \ref{Band_structure}
is complete in $\O(\L_{kn})$.

(ii) Under certain mild non-degeneracy conditions, for any generalized cluster structure  there exists a compatible quadratic Poisson structure (see \cite[Proposition 2.5]{GSVDouble} for details). We expect this compatible Poisson structure to coincide with a (perhaps, modified) natural Poisson structure on the space of periodic finite difference operators introduced in
\cite{Anton} for the proof of complete integrability of generalized pentagram maps.
\end{remark}

Details and proofs of assertions mentioned in the above remark will be considered in a separate publication.

\begin{figure}[ht]
\begin{center}
\includegraphics[width=5cm]{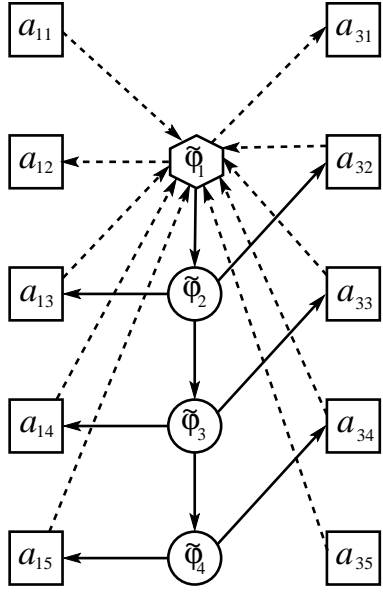}
\end{center}
\caption{Quiver $Q_{25}$}
\label{Q25}
\end{figure}

Let us examine the case $k=2$ in more detail. By \cite[Theorem 2.7]{CS}, the finite type classification for generalized cluster structures coincides with that for usual cluster structures. Consequently, $\GCC(\Sigma_{2n})$ is of type $C_{n-1}$.
In \cite{YZ}, every cluster structure of finite type with principal coefficients was given a geometric realization in the 
ring of regular of functions on a reduced double Bruhat cell corresponding to a Coxeter element of the Weyl group and its inverse. In the $A_n$ case, this double Bruhat cell consists of tridiagonal matrices $A$ in $SL_{n+1}$ with nonzero 
off-diagonal entries and with subdiagonal entries normalized to be equal to~$1$.
Then \cite[Theorem 1.1]{YZ} shows that the set of mutable cluster variables in such a realization coincides with the set of all dense principal minors of $A$.

We have the following analogue of \cite[Theorem 1.1]{YZ}.

\begin{proposition}
\label{YZ}
The set of mutable cluster variables in $\GCC\left (\Sigma_{2n} \right )$ coincides with the set of all distinct dense
 principal minors of $L\in\L_{2n}$ of size less than $n$.
\end{proposition}

\begin{proof} Since $\GCC\left (\Sigma_{2n} \right )$ is a generalized cluster structure of type $C_{n-1}$, the number of mutable cluster variables is $n(n-1)$, that is, the number of almost positive roots in $C_{n-1}$.  
Since this is also the number of distinct dense principal minors of 
$L\in\mathcal{L}_{2n}$ of size less than $n$, we only need to show that every such minor appears as a cluster variable in 
$\GCC\left (\Sigma_{2n} \right )$. In the spirit of \cite{YZ}, we denote by $x_{[i,j]}$ the dense principal minor of $L$ with diagonal entries $a_{2i}, a_{2,i+1},\ldots,a_{2,j-1}, a_{2j}$, where either $1\le i\le j\le n$, $(i,j)\ne (1,n)$, or 
$1\le j < i-1\le n-1$.

The initial cluster variables $\tilde\phhi_1,\ldots, \tilde\phhi_{n-1}$ are minors $x_{[i,n]}$, $i=2,\ldots,n$, contained 
in an $(n-1)\times (n-1)$ tridiagonal matrix $\Phi_{[1,n-1]}^{[1,n-1]}$. If we treat, temporarily, $\tilde\phhi_1$ as a 
frozen variable, $\tilde\varphi_2,\ldots, \tilde\varphi_{n-1}$ form an initial cluster of a cluster structure of finite  
type $A_{n-2}$, whose set of mutable cluster variables is the collection
$x_{[i,j]}$, $2\leq i \le j \leq n-1$, according to \cite[Theorem 1.1]{YZ}. (In \cite{YZ}, the corresponding tridiagonal matrix is normalized to have determinant $1$, and also all the subdiagonal entries are equal to one, however, the calculation needed to obtain the desired result goes through without any modifications). 

Next, we perform a generalized mutation from our initial cluster in direction $1$ using \eqref{longidXYtri21}. 
We claim that $\tilde\phhi_1^*$ is equal to $x_{[3,1]}$. Clearly, the degree of $\tilde\phhi_1^*$ in matrix entries of $L$ is equal to $n-1$. By \eqref{phi_star},
\begin{equation*}
\phhi_1^* = (-1)^{n+1} \det K^*(U^{-1}; e_1,v_\gamma)  
\left ( \det \bar Y\right )^2,
\end{equation*}
and so $\tilde\phhi_1^*=\phhi_1/\det\bar Y$ is proportional to the numerator of $\det K^*(U^{-1}; e_1,v_\gamma)$ viewed as rational function in terms of entries of $L$ with a coefficient that is a monomial in $a_{3j}$, $j=1,\ldots, n$. Since the degree of 
$x_{[3,1]}$ is $n-1$, we only need to show that $\det K^*(U^{-1}; e_1,v_\gamma)$ is proportional to  $x_{[3,1]}$. 

Recall that for band matrices $\gamma=0$, and so $v_\gamma$ defined in Lemma \ref{detphi2} is equal to 
$U e_2$. Then $w$ in Proposition \ref{long_identity} becomes 
$w =\left  [ -u_{22}, u_{12}\right ]$,
 and we obtain $\det K^*(U^{-1}; e_1,v_\gamma)= -u_{12}$. By \eqref{bandU}, 
\[
u_{12} = \frac{(-1)^{n+1} a_{31}}{\det Y} \left (a_{11}\det Y^{\hat 1\widehat{n-1}}_{\hat 1\widehat 2}-
a_{21} \det Y^{\hat 1\widehat n}_{\hat 1\widehat 2}  \right) = \frac{(-1)^{n}x_{[3,1]} }{\det Y}, 
\]
 and hence $\tilde\phhi_1^*=x_{[3,1]}$; here in the last equality we used the expansion of $x_{[3,1]}$ with respect to the last row. 

After the generalized mutation, the quiver is transformed as follows: all edges incident to the special vertex change
 direction, edges pointing from the vertex corresponding to $ \tilde\phhi_2$ to vertices corresponding to $a_{13}$ and $a_{32}$ disappear, but new edges appear instead pointing to  $\tilde\phhi_2$ from frozen vertices corresponding to $a_{11}, a_{14},\ldots, a_{1n}$ and  $a_{33},\ldots, a_{3n}$ (cf.~Fig.~\ref{Q25}). It is easy to check via \eqref{jacobi} for the 
submatrix of $L$  obtained from $\Phi_{[1,n-1]}^{[1,n-1]}$ by cyclically shifting second indices of all entries $a_{ij}$ up 
by~$1$ that mutation at the vertex $(2,2)$ transforms
$\tilde\phhi_2$ to $x_{[4,1]}$. Similarly, consequent mutations at the vertices $(3,2), (4,2),\dots,(n-1,2)$ transform
each $\tilde\phhi_i$ to $x_{[i+2,1]}$, $i=3,\dots,n-2$, and $\tilde\phhi_{n-1}$ to $x_{[1,1]}$. Moreover, the resulting quiver coincides with the initial one. 
Clearly, we can perform a similar shift operation $n-2$ more times and recover the rest of functions $x_{[i,j]}$ as cluster variables.
\end{proof}

\begin{remark} Proposition \ref{YZ} provides a geometric realization of generalized cluster structures of finite type $C_n$.
We should mention that generalized cluster algebras of this type but with {\em constant\/} exchange coefficients have been recently considered in \cite{Gl} in the context of study of representations of the quantum loop algebra of $sl_2$ at roots of unity, and in \cite[section 9]{LV}, where they were realized as {\em Caldero-Chapoton algebras\/} associated with a special triangulation of a polygon with one orbifold point.
\end{remark}

\section{Example 3: Exotic generalized cluster structure on $GL_6$.}
\label{exotic}

In~\cite{GSVMMJ} we initiated the study of cluster structures in the ring of regular functions on $GL_n$ compatible with
R-matrix Poisson--Lie brackets. Such brackets are classified by Belavin--Drinfeld triples 
$\bfG=(\Gamma_1, \Gamma_2, \gamma: \Gamma_1 \to \Gamma_2)$, where $\Gamma_1$ and $\Gamma_2$ are subsets of the set of
positive simple roots in the $A_{n-1}$ root system and $\gamma$ is a nilpotent isometry 
(see~\cite{GSVMMJ} for details).  The cluster structures corresponding to non-empty Belavin--Drinfeld triples
are called {\em exotic}. In~\cite{GSVPleth} we treated the
subclass of Belavin--Drinfeld triples that we called aperiodic. The first instance of a periodic Belavin--Drinfeld triple 
occurs for $n=6$ with the triple $\bfG$ given by
\begin{equation}\label{Gamma}
 \Gamma_1 = \{\alpha_1,\alpha_5\},\ \Gamma_2 = \{\alpha_2,\alpha_4\},\quad \gamma(\alpha_1) = \alpha_2, \
\gamma(\alpha_5) = \alpha_4.
 \end{equation}

It will  be convenient to denote elements of $D(GL_6)$ by $(R,S)$. 
Following the construction described in \cite{GSVPleth}, we consider a collection of matrices 
\[
\L_{\bfG}(R,S)=\{R, R_{[5,6]}^{[1,2]}, R_{[3,6]}^{[1,4]}, R_{[4,6]}^{[1,3]},  S_1^6, S_{[1,5]}^{[2,6]}, S^{[4,6]}_{[1,3]}, 
L_1(R,S), L_2(R,S)\},
\] 
where
$L_1=L_1(R,S)$, $L_2=L_2(R,S)$ both have a form \eqref{shape}, see Fig.~\ref{twomat1}.
 
\begin{figure}[ht]
\begin{center}
\includegraphics[width=10cm]{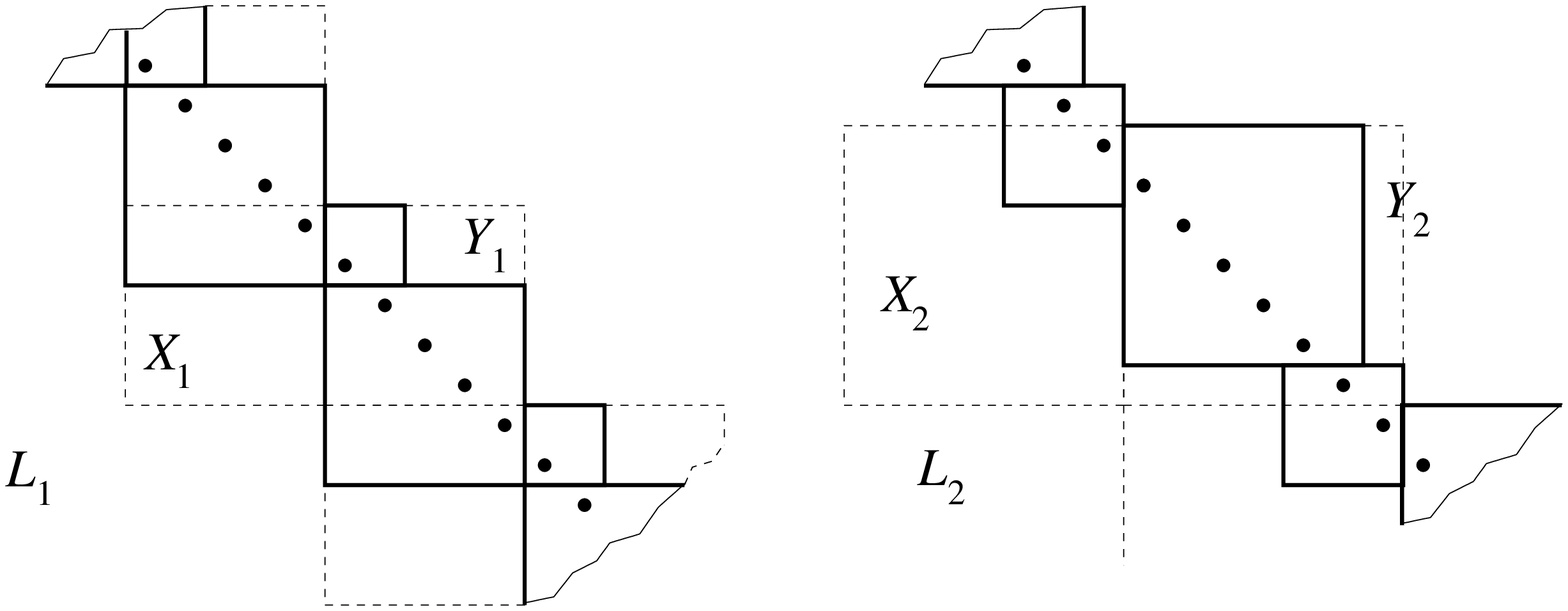}
\end{center}
\caption{Matrices $L_1$ and $L_2$}
\label{twomat1}
\end{figure}

Here  $2\times 2$  and $5\times 5$ blocks featured in $L_1$ are submarices
$R_{[5,6]}^{[1,2]}$ and $S_{[1,5]}^{[2,6]}$, while $3\times 3$  and $6\times 6$ blocks featured in $L_2$ are 
$S_{[1,3]}^{[4,6]}$ and $R$. $L_1$ is $5$-periodic, $L_2$ is $7$-periodic and each has one inner diagonal, which corresponds to 
$k=2$ in~\eqref{shape}. Overlaps between blocks in $L_1$, $L_2$ are prescribed by $\bfG$ (see \cite{GSVPleth} for details).

For $L_1(R,S)$ we choose  
\[
X_1 = \left [
\begin{array}{c}
 S_{[4,5]}^{[2,6]}\\ 0_{3\times 5} 
\end{array}
\right ],\quad
Y_1= \left [
\begin{array}{cc}
R_{[5,6]}^{[1,2]} & 0_{2\times 3} \\
S_{[1,3]}^{[2,3]} & S_{[1,3]}^{[4,6]} 
\end{array}
\right ],
\]
which corresponds to $a=2$, $b=0$,
while for $L_2(R,S)$ we choose  
\[
X_2 = \left [
\begin{array}{cc}
0_{2\times 4} & S^{[4,6]}_{[2,3]}\\ 0_{5\times 4} & 0_{5\times 3}
\end{array}
\right ],\quad
Y_2= \left [
\begin{array}{ccccccc}
r_{11} & r_{12} & r_{13} & r_{14} & r_{15} & r_{16} & 0 \\
r_{21} & r_{22} & r_{23} & r_{24} & r_{25} & r_{26} & 0 \\
r_{31} & r_{32} & r_{33} & r_{34} & r_{35} & r_{36} & 0 \\
r_{41} & r_{42} & r_{43} & r_{44} & r_{45} & r_{46} & 0 \\
r_{51} & r_{52} & r_{53} & r_{54} & r_{55} & r_{56} & 0 \\
r_{61} & r_{62} & r_{63} & r_{64} & r_{65} & r_{66} & 0 \\
0 & 0 & 0 & 0 & s_{14} & s_{15} & s_{16}
\end{array}
\right ],
\]
which corresponds to $a=6$, $b=4$.
Thus \eqref{Phi} results in 
\[
\Phi_1=\left [
\begin{array}{ccccc}
r_{61} & r_{62} & 0 & 0 & 0 \\
s_{12} & s_{13} & s_{14} & s_{15} & s_{16}  \\
s_{22} & s_{23} & s_{24} & s_{25} & s_{26}  \\
s_{32} & s_{33} & s_{34} & s_{35} & s_{36}  \\
s_{52} & s_{53} & s_{54} & s_{55} & s_{56}  
\end{array}
\right ]
\]
and
\[
\Phi_2=\left [
\begin{array}{ccccccc}
r_{21} & r_{22} & r_{23} & r_{24} & r_{25} & r_{26} & 0 \\
r_{31} & r_{32} & r_{33} & r_{34} & r_{35} & r_{36} & 0\\
r_{41} & r_{42} & r_{43} & r_{44} & r_{45} & r_{46} & 0\\
r_{51} & r_{52} & r_{53} & r_{54} & r_{55} & r_{56} & 0\\
r_{61} & r_{62} & r_{63} & r_{64} & r_{65} & r_{66}& 0 \\
0 & 0 & 0 & 0 & s_{14} & s_{15} & s_{16}\\
0 & 0 & 0 & 0 & s_{34} & s_{35} & s_{36}
\end{array}
\right ]. 
\]
Consequently, \eqref{detXYcor} yields
\begin{equation}
\label{detXYex}
\begin{split}
\det\left (\lambda Y_1 + \mu X_1\right) &=  \lambda^3 \left (\det S_{[1,5]}^{[2,6]}\mu^2  + c_{11}(R,S) \lambda\mu  + 
\det S_{[1,3]}^{[4,6]}\det R_{[5,6]}^{[1,2]} \lambda^2  \right ),\\
\det\left (\lambda Y_2 + \mu X_2\right) &=  \lambda^5 \left ( \det S_{[1,3]}^{[4,6]}\det R_{[3,6]}^{[1,4]}\mu^2  + 
c_{21}(R,S) \lambda\mu  +  s_{16}\det R \lambda^2 \right ). 
\end{split}
\end{equation}

Let us denote the functions associated with $\Phi_1$, $\Phi_2$ via \eqref{phis} by $\phhi_{1i}$, $1\le i\le 5$, 
and $\phhi_{2i}$, $1\le i\le 7$, respectively. Taking into account that 
\[
\det \bar Y_1=r_{62} \det S_{[1,3]}^{[4,6]},\qquad  
\det \bar Y_2=s_{16} \det R_{[2,6]}^{[2,6]}, 
\]
we obtain from~\eqref{longidXY} and~\eqref{detXYex}
\begin{equation}
\label{GammaId}
\begin{split}
\phhi_{11}\varphi_{11}^*&= 
\det S_{[1,5]}^{[2,6]} \det S_{[1,3]}^{[4,6]} r_{62}^2 + c_{11}(R,S) r_{62}\phhi_{12} + 
\det R_{[5,6]}^{[1,2]} \phhi_{12}^2 , \\
\phhi_{21}\phhi_{21}^*&=  s_{16}\det S_{[1,3]}^{[4,6]}\det R_{[3,6]}^{[1,4]} \left ( \det R_{[2,6]}^{[2,6]}\right )^2 + 
c_{21}(R,S)\det R_{[2,6]}^{[2,6]}\phhi_{22}  +  \det R \phhi_{22} ^2,
\end{split}
\end{equation}
where $\phhi_{11}^*$ and $\phhi_{21}^*$ are polynomial in the entries of $R$, $S$.

Recall that the family 
\[
\FFF_{n}^{\rm st}=\left \{ \{g_{ij}(R)\}_{1\le j\le i\le n},\ \{h_{ij}(R)\}_{1\le i < j\le n}\right\}
\] 
with $g_{ij}$ and $h_{ij}$ defined in the previous section 
is a cluster for the standard cluster structure on $GL_n$ that has a property that for every pair $i,j$ of indices between~$1$ and~$n$ there is a unique function in $\FFF_{n}^{\rm st}$ represented by a minor whose upper left entry is $r_{ij}$. These functions are attached in a natural way to vertices of the corresponding quiver, $Q_{n}^{\rm st}$, that form an 
$n\times n$ grid with all the vertices in the first row and column frozen. The edges $(i,j)\to (i+1,j+1)$, $(i+1,j)\to (i,j)$ and $(i,j+1)\to (i,j)$ form a mesh of consistently oriented triangles (except that edges between frozen variables are ignored). 

Let now $\FFF_\bfG$ be the family of functions that consists of all distinct dense trailing minors of matrices that comprise 
$\L_\bfG(R,R)$.  Alternatively, we can describe $\FFF_\bfG$ as
\[
\begin{split}
\FFF_\bfG =&\left (\FFF_{6}^{\rm st}\setminus 
\left \{ \{g_{i+1,i}(R)\}_{1\le i\le 5}, \ g_{61}(R), \{h_{i,i+2}(R)\}_{1\le i\le 4},h_{15}(R),h_{26}(R)\right\}
\right )\\
&\cup \left\{   \{\phhi_{1i}(R,R)\}_{1\le i\le 4}, \{\phhi_{2i}(R,R)\}_{1\le i\le 6} \right \}.
\end{split}
\] 
Note that $\FFF_\bfG$ contains only $34$ functions in contrast with $\FFF_{6}^{\rm st}$ which contains $36$. Specifically, none of the functions in $\FFF_\bfG$ is represented as a minor whose upper left entry is $r_{26}$ or $r_{46}$. All other 
$r_{ij}$ do appear in this way, and so we attach them to the corresponding nodes of a $6\times 6$ grid that will serve as
the vertex set of the quiver $Q_\bfG$ depicted in Fig.~\ref{qsl6}. Here the white vertices denote functions in the intersection $\FFF_\bfG\cap\FFF_{6}^{\rm st}$, the ones with the vertical filling refer to $\phhi_{1i}$, and the 
ones with the diagonal filling, to  $\phhi_{2i}$. The special vertices $(6,1)$ and $(2,1)$ correspond to  
$\phhi_{11}$ and $\phhi_{21}$, respectively. Strings of exchange coefficients attached to these vertices are 
$(1,c_{11}(R,R),1)$ and $(1,c_{21}(R,R),1)$, respectively. These are the only nontrivial strings in the set of strings 
$\P_\bfG$ that we associated with $Q_\bfG$ and $\FF_\bfG$. The corresponding generalized exchange relations are obtained from~\eqref{GammaId}:
\begin{equation*}
\begin{split}
\phhi_{11}\phhi_{11}^*&= 
h_{12} h_{14} g_{62}^2 + c_{11}(R,R) g_{62}\phhi_{12} + g_{51} \phhi_{12}^2 , \\
\phhi_{21}\phhi_{21}^*&=  h_{16}h_{14}g_{31} g_{22}^2 + c_{21}(R,R)g_{22}\phhi_{22}  +  g_{11} \phhi_{22} ^2.
\end{split}
\end{equation*}
\begin{figure}[ht]
\begin{center}
\includegraphics[width=8cm]{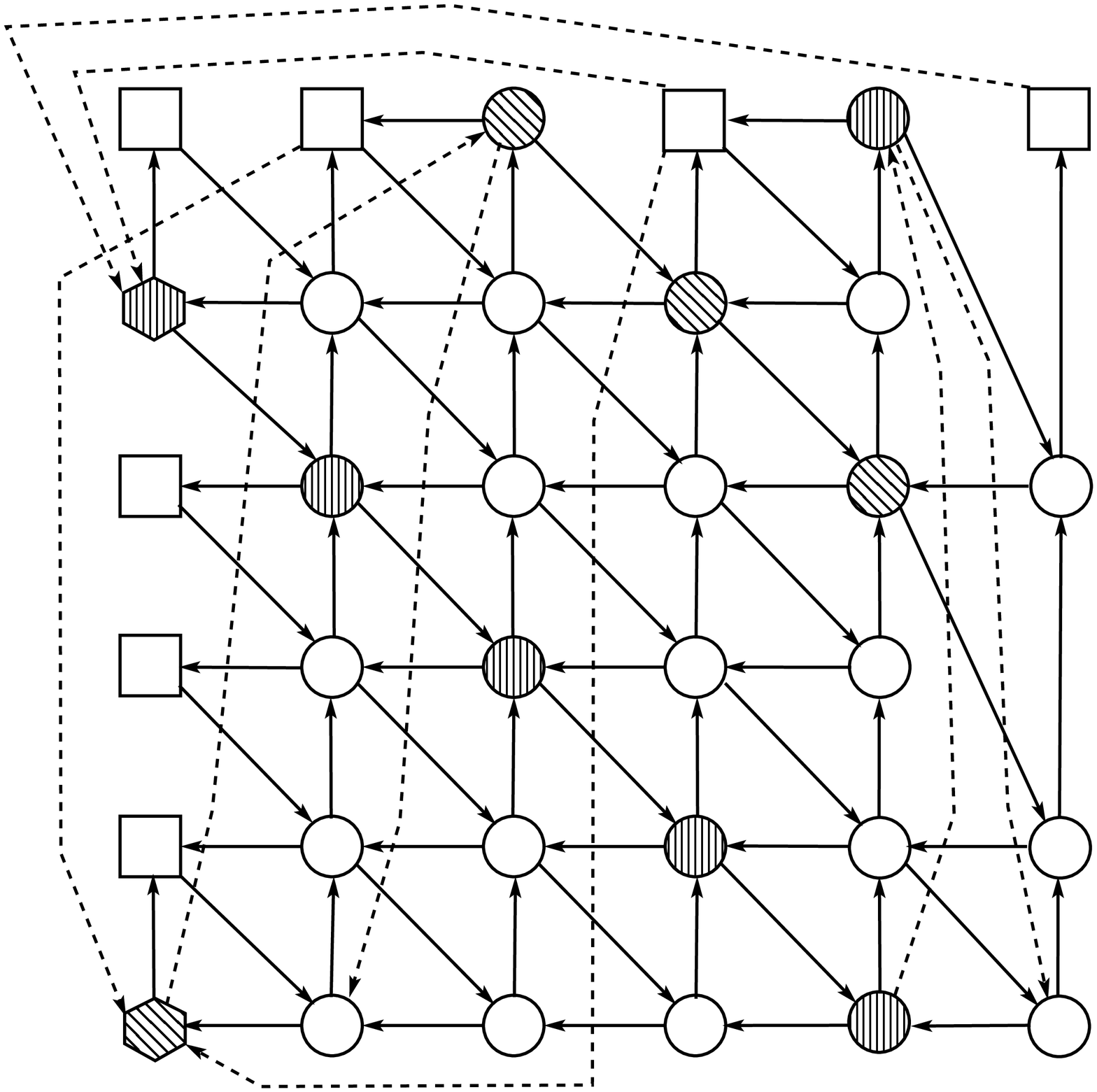}
\end{center}
\caption{Quiver $Q_\bfG$}
\label{qsl6}
\end{figure}

\begin{proposition}
\label{exex}
The seed $\Sigma_{\bfG}=(\FFF_{\bfG},Q_{\bfG},\P_{\bfG})$ defines 
a regular complete generalized cluster structure  
in the ring of regular functions on $GL_{6}$. This structure is compatible with the Poisson--Lie bracket 
$\Poi_\bfG$ specified by $\bfG$ given by~\eqref{Gamma}.
\end{proposition}

\begin{proof} The proof is based on lengthy calculations, some of them straightforward, some {\em ad hoc\/}, 
and some relying on symbolic computations using MAPLE. In particular, the proof of regularity relies on Theorem~\ref{main} 
and identities \eqref{jacobi},\eqref{pluck}, \eqref{pluckpluck}, just like in the proofs of Theorem \ref{Band_structure} and Theorem \ref{structure}. The proof of compatibility of $\Sigma_{\bfG}$ with $\Poi_\bfG$ is MAPLE assisted.  To prove completeness, we constructed sequences of mutations that recover matrix entries $x_{51}, x_{61}$ and $x_{ij}$, $i=1,3,4,5,6$, 
$j=2,3,4,6$, as cluster variables. For each of the remaining matrix entries, we found two Laurent polynomial expressions 
of the form $\frac{M}{f}$, where $M\in \UU (\Sigma_{\bfG})$ and $f$'s entering two expressions for the same matrix element  are coprime cluster variables. By \cite[Lemma 8.3]{GSVMem}, this guarantees that matrix entries in question belong to 
$\UU (\Sigma_{\bfG})$. We omit the details of the proof since the general case of generalized cluster structures associated with Poisson brackets that arise in the Belavin--Drinfeld classification will be treated in a follow-up to 
\cite{GSVPleth}.
\end{proof}

Note that the choice of the periodic staircase structure for the matrices $L_1$ and $L_2$ is not unique. Each one of them admits one more such structure, as shown in Fig.~\ref{twomat2}.  

\begin{figure}[ht]
\begin{center}
\includegraphics[width=10cm]{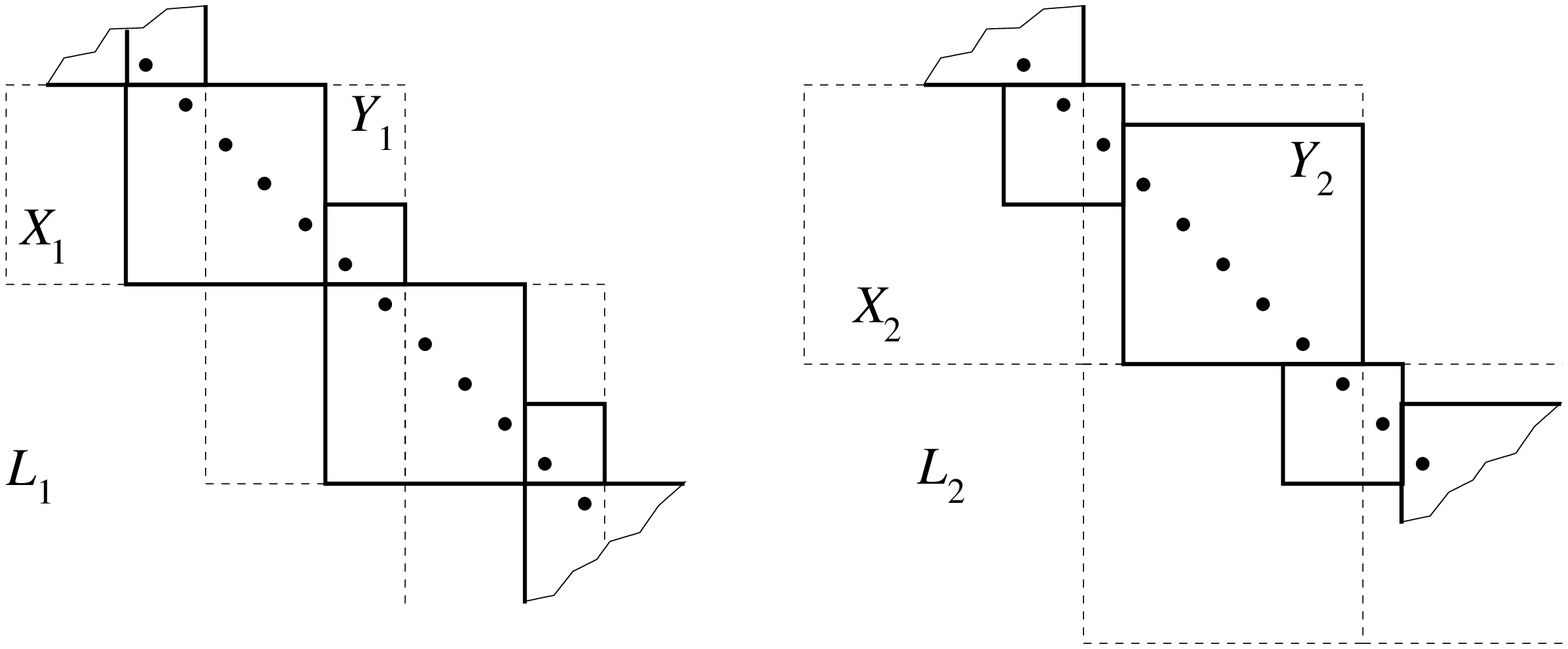}
\end{center}
\caption{Another partition of matrices $L_1$ and $L_2$}
\label{twomat2}
\end{figure}

Any pair of choices presented on Figs.~\ref{twomat1} and~\ref{twomat2} gives rise to a regular complete generalized cluster 
structure compatible with the Poisson--Lie structure $\Poi_\bfG$. It is interesting to investigate whether the seeds thus obtained are mutation equivalent.

\section{Properties of core minors}\label{corprop}

\subsection{Expressing core minors via $U$}\label{detlemmas}

\begin{proof}[Proof of Lemma~\ref{detphi1}] Using block-column operations, we obtain
\[
\phhi_1 =  \left ( \det Y \right )^{k-1} 
\det \left  [
\begin{array}{ccccc}
\left ( \one_n\right )_{[2,n]} & & & & \\
\left [ \begin{array}{c} W_{[2,a]}\\ 0 \end{array} \right ] & \left ( \one_n\right )_{[2,n]} & & & \\
 & \ddots & \ddots & & \\
  & & \left [ \begin{array}{c} W_{[2,a]}\\ 0 \end{array} \right ] & \left ( \one_n\right )_{[2,n]} & \\
  & & & W_{[2,a]} & Y_{[2,a]}^{[1,b]}
\end{array}
\right ].
\]
In the second determinant above there are rows containing a single non-zero entry equal to $1$. Removing these rows and corresponding columns, we can be further re-write it as
\begin{align}\nonumber
 &\ \varepsilon\det \left  [
\begin{array}{ccccc}
W_{[2,a]}^{[1]} & \left ( \one_a\right )_{[2,a]} & & & \\
 & \ddots & \ddots & & \\
  & & W_{[2,a]}^{[1,a]} &\left ( \one_a\right )_{[2,a]}  & \\
  & & & W_{[2,a]}^{[1,a]} & Y_{[2,a]}^{[1,b]}
\end{array}
\right ]\\
\nonumber
= &\ \varepsilon \det Y_2 
\det \left  [
\begin{array}{ccccc}
W_{[2,a]}^{[1]} & \left ( \one_a\right )_{[2,a]} & & & \\
 & \ddots & \ddots & & \\
  & & W_{[2,a]}^{[1,a]} &\left ( \one_a\right )_{[2,a]}  & \\
  & & & \left [ \begin{array}{cc} U_{[2,k]} &\star\\ W_{21} & W_{22}\end{array} \right ]  & \left [ \begin{array}{c} 0\\ \one_b \end{array} \right ] 
\end{array}
\right ]\\
\label{aux1}
= &\ \varepsilon \det Y_2 
\det \left  [
\begin{array}{cccc}
W_{[2,a]}^{[1]} & \left ( \one_a\right )_{[2,a]} & &  \\
 & \ddots & \ddots &  \\
  & & W_{[2,a]}^{[1,a]} &\left ( \one_a\right )_{[2,a]}   \\
  & & & \left [ \begin{array}{cc} U_{[2,k]} &\star \end{array} \right ]  
\end{array}
\right ]
\end{align} 
with $\varepsilon=(-1)^{n-1+(n-a)([k/2]-1)}$.
The first equality above is obtained by multiplying the last block row on the left by 
$\left [\begin{array}{cc} \one_{k} & - Y_1 Y_2^{-1} \\0 & \one_{b}\end{array} \right ]_{[2,a]}^{[2,a]}$. 

Next, transform the matrix featured in \eqref{aux1} by
multiplying the last block column on the right by $W^{[1,a]}$ and subtracting it from the previous one, 
then multiplying the $(k-2)$nd block column on the right by $W^{[1,a]}$ and subtracting it from the
$(k - 3)$rd one, etc., finally, multiplying the $2$nd block column by $W^{[1]}$ and subtracting
it from the first block column. The resulting matrix equals
\begin{equation}
\label{aux2}
 \left  [
\begin{array}{cccc}
0 & \left ( \one_a\right )_{[2,a]} & &  \\
 0 & 0 & \ddots &  \\
  & &  &\left ( \one_a\right )_{[2,a]}   \\
V_{k-2}^{[1]}  &\cdots  & V_1 & \left [ \begin{array}{cc} U_{[2,k]} &\star \end{array} \right ]  
\end{array}
\right ]
\end{equation}
with $V_i = (-1)^i \left [ U_{[2,k]}\  \star \right ] \left ( W^{[1,a]} \right )^i$ for $i=1,\ldots, k-2$.
Note that for $j=1,\ldots, k$ we have
\begin{align}
\nonumber
V_i^{[j]} & = (-1)^i \left (  
\left [\begin{array}{cc} \one_{k} & - Y_1 Y_2^{-1} \\0 & \one_{b}\end{array} \right ] \left ( W^{[1,a]} \right )^{i+1}
 \right )^{[j]}_{[2,k]}\\
 \nonumber
&= (-1)^i \left (    
\left [\begin{array}{cc} \one_{k} & - Y_1 Y_2^{-1} \\0 & \one_{b}\end{array} \right ] \left ( W^{[1,a]} \right )^{i+1}
\left [\begin{array}{cc} \one_{k} & -Y_1 Y_2^{-1} \\0 & \one_{b}\end{array} \right ]
 \right )^{[j]}_{[2,k]},
\end{align}
and hence \eqref{aux0} implies
\begin{equation}\label{aux3}
V_i^{[j]} = (-1)^i \left (U^{i+1}\right )^{[j]}_{[2,k]}.
\end{equation}
This means that the determinant of the matrix in \eqref{aux2} equals
\[
\varepsilon' \det \left [ V^{[1]}_{k-2}\ldots V^{[1]}_{1} U^{[1]}_{[2,k]} \right ]  = 
(-1)^{k-1}\varepsilon' \det \left [ U^{k-1}e_1 \ldots U^{2}e_1\  U e_1 \ e_1\right ]
\]
with $\varepsilon'=(-1)^{(k-1)(k-2)/2+(a-1)([k/2]-1)}$, and the claim of the lemma follows since
$(-1)^{k-1}\varepsilon\varepsilon'=(-1)^{k(k-1)/2+(n-1)[k/2]}=\varepsilon_1$.
\end{proof}

\begin{proof}[Proof of Lemma~\ref{detphi2}] Define $\bar W$ via $X^{[2,n]}\bar Y^{-1}=\left[\begin{array}{c} \bar W\\ 0_{(n-a)\times(n-1)} \end{array}\right ]$.
We proceed as in the proof of Lemma \ref{detphi1} and get
\begin{align*}\nonumber
\phhi_2 &= 
 \left ( \det Y \right )^{k-2}  \det\bar Y 
 \det \left  [
\begin{array}{ccccc}
 \left ( \one_{n-1}\right )_{[3,n]}  & & & & \\
 \left [ \begin{array}{c} \bar W_{[2,a]}\\ 0 \end{array} \right ] & \left ( \one_n\right )_{[2,n]} & & & \\
 & \ddots & \ddots & & \\
  & &  \left [ \begin{array}{c}  W_{[2,a]}\\ 0 \end{array} \right ] & \left ( \one_n\right )_{[2,n]}   & \\
  & & & W_{[2,a]}^{[1,a]} & Y_{[2,a]}^{[1,b]}
\end{array}
\right ]\\
\nonumber
&= -\varepsilon \left ( \det Y \right )^{k-2}  \det\bar Y \det Y_2
\det \left  [
\begin{array}{cccc}
\bar W_{[2,a]}^{[1]} & \left ( \one_a\right )_{[2,a]} & &  \\
 & \ddots & \ddots &  \\
  & & W_{[2,a]}^{[1,a]} &\left ( \one_a\right )_{[2,a]}   \\
  & & & \left [ \begin{array}{cc} U_{[2,k]} &\star \end{array} \right ] 
\end{array}
\right ],
\end{align*}
where $\varepsilon$ is the same as in \eqref{aux1}. Similarly to the proof of Lemma \ref{detphi1}, this yields
\[
 \phhi_2= (-1)^{k-1}\varepsilon_1 \left ( \det Y \right )^{k-2}  \det\bar Y \det Y_2  \det \left [ w \  U^{k-2}e_1 \ldots U^{2}e_1\  U e_1 \ e_1\right ],
\] 
where $w = (-1)^{k-2}\left [ \begin{array}{cc} U_{[2,k]} &\star\end{array}\right]\left( W^{[1,a]}\right )^{k-3}\bar W^{[1]}$. 

Next, factor $Y$ as $Y=
\left [\begin{array}{cc} 1 & \star   \\0 & \one_{n-1}\end{array} \right ]
\left [\begin{array}{cc}  \star & 0 \\ \star & \bar Y\end{array} 
\right]$. Then
\[ X Y^{-1}=
\left [\begin{array}{cc} 0\ X^{[2,n]} \end{array} \right ]
\left [\begin{array}{cc}  \star & 0 \\ \star & \bar Y^{-1} \end{array} 
\right ] 
\left [\begin{array}{cc} 1 & \star   \\0 & \one_{n-1}\end{array} \right ]
\]
implies $W =\left [\begin{array}{cc} W^{[1]}\ \bar W \end{array} \right ]
\left [\begin{array}{cc} 1 & \star   \\0 & \one_{n-1}\end{array} \right ]$. 
Consequently, $\bar W^{[1]} = W^{[2]} + \gamma  W^{[1]}$, where $\gamma$ is given by \eqref{gamma}.
It remains to use \eqref{aux3} to get $w = (-1)^{k-2} U^{k-1} (e_2 + \gamma e_1) 
= (-1)^{k-2} U^{k-2} v$.
\end{proof}

\begin{proof}[Proof of Lemma~\ref{detphilast}]
Similar considerations show that
\begin{align*}\nonumber
\phhi_{n+1} & =\left ( \det Y \right )^{k-2}  
 \det \left  [
\begin{array}{ccccc}
 \left ( \one_{n}\right )_{[3,n]}  & & & & \\
 \left [ \begin{array}{c} W_{[2,a]}\\ 0 \end{array} \right ] & \left ( \one_n\right )_{[2,n]} & & & \\
 & \ddots & \ddots & & \\
  & &  \left [ \begin{array}{c}  W_{[2,a]}\\ 0 \end{array} \right ] & \left ( \one_n\right )_{[2,n]}   & \\
  & & & W_{[2,a]} & Y_{[2,a]}^{[1,b]}
\end{array}
\right ]\\
\nonumber
&= \bar\varepsilon \left ( \det Y \right )^{k-2}   \det Y_2
\det \left  [
\begin{array}{ccccc}
 W_{[2,a]}^{[1,2]} & \left ( \one_a\right )_{[2,a]} & & \\
 & \ddots & \ddots &  \\
  & & W_{[2,a]}^{[1,a]} &\left ( \one_a\right )_{[2,a]}   \\
  & & & \left [ \begin{array}{cc} U_{[2,k]} &\star \end{array} \right ] 
\end{array}
\right ]
\end{align*}
with $\bar\varepsilon=(-1)^{(n-a)([k/2]-1)}$, similarly to \eqref{aux1}. This leads to an analog of \eqref{aux2}, which
yields
\begin{align*}
\phhi_{n+1}&=\bar\varepsilon\bar\varepsilon'  
\left ( \det Y \right )^{k-2} \det Y_2  \det \left [U^{k-2}e_1\ U^{k-2}e_2\ U^{k-3}e_1  \ldots U^{2}e_1\  U e_1 \ e_1\right ]\\
&= \varepsilon_{n+1}
\left ( \det Y \right )^{k-2} \det Y_2  \det \left [U^{k-2}e_1\ U^{k-3}v_\gamma\ U^{k-3}e_1 \ldots U^{2}e_1\  U e_1 \ e_1\right],
\end{align*} 
where $\bar\varepsilon'=(-1)^{(a-1)([k/2]-1)+k(k-1)/2-1}$.
\end{proof}

\subsection{Irreducibility of core minors}\label{varirr}

\begin{proof}[Proof of Lemma~\ref{coreirr}]
For $X, Y$ given by \eqref{shape}, we say that $\phhi_1$ is of {\it type\/} $(n,k,b)$. The three parameters satisfy
conditions 
\begin{equation*}\label{cdcond}
n\ge k+b, \qquad k\ge 2, \qquad b\ge 0. 
\end{equation*}
\noindent The proof of irreducibility is based on induction on all the parameters.

For type $(n,2,0)$, the irreducibility of $\phhi_1$ is a well-known fact, since the corresponding core is an $n\times n$ matrix of independent variables. For type $(3,3,0)$ we have
\[
\phhi_1=\begin{vmatrix} 
y_{21} & y_{22} & y_{23} & 0 & 0 & 0\\
y_{31} & y_{32} & y_{33} & 0 & 0 & 0\\
x_{21} & x_{22} & x_{23} & y_{21} & y_{22} & y_{23}\\
x_{31} & x_{32} & x_{33} & y_{31} & y_{32} & y_{33}\\
0 & 0 & 0 & x_{21} & x_{22} & x_{23}\\
0 & 0 & 0 & x_{31} & x_{32} & x_{33}
\end{vmatrix},
\]
and its irreducibility can be verified by direct observation. In a similar way one can treat the case $(4,3,0)$.

Let now $\phhi_1$ be of type $(n,3,0)$ with $n>4$. Note that $\phhi_1$ is a homogeneous polynomial of degree $2$ in 
each variable. Assume that $\phhi_1=PQ$, then both $P$ and $Q$ are homogeneous. Let $y=y_{41}$. Note that the coefficient 
$c_y$ at $y^2$ in $\phhi_1$ equals $\pm\phhi_1(X^{[2,n]},Y_{[1,3]\cup[5,n]}^{[2,n]})$; the latter is of type $(n-1,3,0)$, 
and hence is irreducible by induction. Consequently, $P=c_yy^p+o(y^p)$ and $Q=y^q+o(y^q)$ with $p+q=2$. 

Further, $c_y$ has degree $2$ in $z=y_{52}$, and hence $\deg_zP=2$, $\deg_zQ=0$. Similarly to above, the coefficient $c_z$ at
$z^2$ in $\phhi_1$ is an irreducible polynomial of degree $2$ in $y$, and we conclude that $p=2$, $q=0$, and hence $Q$ is a constant.

Let now $\phhi_1$ be of type $(n,k,0)$ with $k>3$. Note that $\phhi_1$ is a homogeneous polynomial of degree $k-1$ in each 
variable. Assume that $\phhi_1=PQ$, then both $P$ and $Q$ are homogeneous. Let $y=y_{21}$. Note that the coefficient at 
$y^{k-1}$ in $\phhi_1$ equals $\pm\psi_1\det Z_1$, where $\psi_1=\phhi_1(\bar X,\bar Y)$ with 
$\bar X=X_{[2,n]}^{[2,n]}$ and $Z_1$ is an $(n-1)\times(n-1)$ matrix 
$\begin{bmatrix} Y_{[k+1,n]}^{[2,n]}\\ X_{[2,k]}^{[2,n]}\end{bmatrix}$. Note that
$\psi_{1}$ is a core determinant of type $(n-1,k-1,0)$, and hence is irreducible by induction, whereas $\det Z_1$ is irreducible as the determinant of a matrix of independent variables. Consequently, either

(i) $P=\psi_{1}y^p+o(y^p)$ and $Q=\pm\det Z_1y^q+o(y^q)$ with $p+q=k-1$, or

(ii)  $P=\psi_{1}\det Z_1y^p+o(y^p)$ and $Q=\pm y^q+o(y^q)$ with $p+q=k-1$.

\noindent In any case, the total degree of $P$ is at least $(k-2)(n-1)+p$, and the total degree of $Q$ is at most $n-1+q$.

Similarly to the treatment of the case $(n,3,0)$ above, we let $z=y_{32}$ and note that $\deg_z\psi_{1}=k-2$, and hence
$\deg_zP\ge k-2$ and $\deg_zQ\le 1$. The same reasoning as above shows that the coefficient at $z^{k-1}$ in
$\phhi_1$ equals $\pm\psi'\det Z'$, where $\psi'$ is a core determinant of type $(n-1,k-1,0)$, and hence irreducible by induction, and $\det Z'$ is the determinant of an $(n-1)\times(n-1)$ matrix of independent variables. Consequently, 
$P=c' z^{p'}+o(p')$ with $p'\ge k-2$, and there are four possibilities for $c'$: 

(a) $c'=\alpha' \psi'\det Z'$,

(b) $c'=\alpha' \psi'$,

(c) $c'=\alpha'\det Z'$,

(d) $c'=\alpha'$,

\noindent where $\alpha'$ is a non-zero constant. 
 
 The last two possibilities are ruled out immediately, since they imply that the total degree of $P$ is at most $n-1+p'$,
which is strictly less than $(k-2)(n-1)+p$. In case (ia) the comparison of the two expressions for the total degree of $P$
gives $(k-2)(n-1)+p=(k-1)(n-1)+p'$, which is equivalent to $p=n-1+p'$, and hence is impossible. Similarly, case (iib) yields
$n-1+p=p'$, which can be satisfied only if $p=0$, $p'=k-1$, and $n=k$. However, $p=0$ in case (iib) means that 
$P=\psi_{1}\det Z_1$, and hence $p'=\deg_zP=k-2$, a contradiction. In the remaining cases (ib) and (iia) we get $p=p'\ge k-2$
and $q=q'\le 1$. 

Assume first that $q=0$. In case (i) we get $Q=\pm \det Z_1$, and, simultaneously, $Q=\alpha'\det Z'$, a contradiction, 
since $Z'=\begin{bmatrix} Y_{[k+1,n]}^{1\cup[3,n]}\\ X_{[2,k]}^{1\cup[3,n]}\end{bmatrix}$. In case (ii) we get that $Q$ is a constant. So, in what follows we assume that $p=k-2$ and $q=1$.

Let now $t\ne y$ be an arbitrary entry in the second row of $Y$ or $X$. Applying the same reasoning as above, we get that 
$\deg_tP=k-2$, $\deg_tQ=1$, and the coefficients at $t^{k-2}$ in $P$ and at $t$ in $Q$ have a similar structure, that is, all of them simultaneously look either as in case (i), or as in case (ii).

Assume that all coefficients are as in case (i). Note that $\psi_{1}$ and all its analogs do not depend on the entries of the second row of $Y$. Consequently, one can write
\[
P=\sum_{j=1}^n y_{2j}^{k-2} \alpha_j\psi_{j} +R+S,
\]
where $\psi_{j}$ are core determinants of type $(n-1,k-1,0)$ depending on submatrices of $X$ and $Y$, 
$\alpha_j$ are non-zero constants with $\alpha_1=1$, $R$ contains all monomials in $P$ that depend on entries in the second row of $Y$ that are not included in the first sum, and $S$ contains only the monomials that do not depend on these entries. Besides, we can write 
\[
Q=\sum_{j=1}^n y_{2j} \beta_j\det Z_j+T,
\] 
where $Z_j$ are $(n-1)\times(n-1)$ matrices built of the entries in the rows $[k+1,n]$ of $Y$ and rows $[2,k]$ of $X$ similarly to $Z_1$, $\beta_j=\pm\alpha_j^{-1}$, and $T$ does not depend on the entries in the second row of $Y$. Consequently, $TS=0$,
since $\phhi_1$ does not contain monomials that do not depend on the entries of the second row of $Y$. Note that $S$ does not vanish since for every entry $t$ in the second row of $X$, $\deg_tP=k-2$ and the coefficient at $t^{k-2}$ in $P$ does not depend on the entries in the second row of $Y$. Therefore, $T=0$ and 
\begin{equation}\label{QinY}
Q=\sum_{j=1}^n y_{2j} \beta_j\det Z_j.
\end{equation}

Let us fix $t=x_{21}$. Recall that $\deg_tQ=1$. Similarly to the treatment of $y$ above, $Q=\bar\beta_1 t\det\bar Y+o(t)$, where $\bar\beta_1$ is a non-zero constant. On the other hand, it follows from \eqref{QinY} that $Q=t\det\bar Z_1 +o(t)$, 
where
\[
 \bar Z_1=\begin{bmatrix} 
\beta_2 y_{22} & \beta_3 y_{23} & \dots & \beta_n y_{2n}\\
y_{k+1,2} & y_{k+1,3} & \dots & y_{k+1,n}\\
\vdots & & & \vdots\\
y_{n2} & y_{n3} & \dots & y_{nn}\\
x_{32} & x_{33} & \dots & x_{3n}\\
\vdots & & & \vdots\\
x_{k2} & x_{k3} & \dots & x_{kn}
\end{bmatrix},
\]
a contradiction.

Assume now that all coefficients are as in case (ii). Then the same treatment as in case (i) leads to
$Q=\sum_{j=1}^n y_{2j} \beta_j$ for some non-zero constants $\beta_j$, and hence the coefficient at $x_{21}$ in $Q$ vanishes,
a contradiction.

To proceed further with the case $b>0$ we will need one more basic type, $(3,2,1)$, in which case
\[
\phhi_1=\begin{vmatrix} 
y_{21} & y_{22} & y_{23} & 0 \\
y_{31} & y_{32} & y_{33} & 0 \\
0 & x_{21} & x_{22} & y_{21} \\
0 & x_{31} & x_{32} & y_{31} 
\end{vmatrix}
\]
is irreducible via direct observation. 

Let now $\phhi_1$ be of type $(n,k,b)$ with $b>0$,
and let $\phhi_1=PQ$. Put $y=y_{21}$ and note that $\deg_y\phhi_1=k$.  
It is easy to see that the coefficient $\bar\psi_2$ at $y^k$ in $\phhi_1$ is itself a core determinant of type $(n-1,k,b-1)$, 
and hence is irreducible by induction. Consequently $P=y^p\bar\psi_2+o(y^p)$ and $Q=\pm y^q+o(y^q)$ with $p+q=k$. In particular, the total degree of $P$ is $(k-1)(n-1)+b-1+p$ and the total degree of $Q$ is $q$.

Similarly, for $z=y_{31}$  we have $P=\alpha z^{p'}\bar\psi_3+o(z^{p'})$ 
and $Q=\beta z^{q'}+o(z^{q'})$ with $p'+q'=k$, where $\bar\psi_3$ is a core determinant of type $(n-1,k,b-1)$ and 
$\alpha\beta=\pm1$ (the opposite case would imply $(k-1)(n-1)b-1+p=q'$, which is impossible). Total degrees of $\bar\psi_2$ 
and $\bar\psi_3$ coincide, so $p=p'$ and $q=q'$. Consequently, $p>0$, since otherwise $P=\bar\psi_2=\alpha\bar\psi_3$, 
a contradiction.

Consider first the case $b=1$ and $n=k+1>3$. Let $t=y_{3n}$, then $\deg_t\phhi_1=\deg_t\bar\psi_2=k-1$, and the coefficient at
$t^{k-1}$ in $\phhi_1$ equals $\bar\psi\det\bar Z$, where $\bar\psi$ is a core determinant of type $(k,k-1,1)$ and hence is 
irreducible by induction, and $\bar Z=[X_{[2,k+1]}^{[1,k]} \ Y_{[2,k+1]}^{[1]}]$ and hence $\det\bar Z$ is irreducible
as the determinant of a matrix of independent variables. Consequently, we have four possibilities similar to (a)--(d) 
above. The last two are ruled out via total degree comparison, since $k^2-k+p>2k$ for $k>2$ and $p>0$.  The second one yields 
$p=1$, in which case $Q=\det\bar Z$ and hence $\deg_yQ=1<k-1=q$, a contradiction. The remaining case yields
$p=k$ and $q=0$, hence $Q$ is a constant.

For $b=1$ and $n>k+1$ take $t=y_{n2}$ and note that $\deg_t\phhi_1=k-1$ and the coefficient $\bar\psi_n$ at $t^{k-1}$ in
$\phhi_1$ is a core determinant of type $(n-1,k,1)$ and hence is irreducible by induction. Moreover, $\deg_y\bar\psi_n=k$,
and hence $p=k$, $q=0$ and $Q$ is a constant.

Finally, for type $(n,k,b)$ with $b>1$ take $u=y_{32}$, then similarly to above, 
$P=\bar\alpha' u^{\bar p'}\bar\psi'+o(u^{\bar p'})$, where $\bar\psi'$ is a core determinant of type $(n-1,k,b-1)$ and hence
is irreducible by induction. Moreover, $\deg_y\bar\psi'=k$, and hence $p=k$, $q=0$ and $Q$ is a constant.
\end{proof}

\begin{proof}[Proof of Lemma~\ref{fnirr}] For $\phhi_1$ this fact is proved in
Lemma~\ref{coreirr} (cp. to the case of type $(n,n,0)$). For other functions $\phhi_i$ the proof is similar. It exploits the
fact that one can find two variables $y$ and $z$  such that the coefficient at the highest degree of the variable in $\phhi_i$
is either an irreducible polynomial or a product of two such polynomials, and that the highest degree of $z$ in $\phhi_i$
equals to the highest degree of $z$ in one of the above two polynomials for $y$. 

In more detail, for $2\le i\le n-1$
one takes $y=x_{n-1,i}$ and $z=x_{n,i+1}$. Then $\deg_y\phhi_i=\deg_z\phhi_i=n-1$ and the coefficients at 
$y^{n-1}$ and $z^{n-1}$
in $\phhi_i$ are equal to $\psi\det Z$, where $\psi$ is $\phhi_1$ for the size $n-1$ and $Z$ is an $(n-i)\times(n-i)$ matrix
of independent variables.

For the case of $\phhi_{pn+i}$, $1\le i\le n-1$, $1\le p\le n-3$, one takes $y=x_{n-1,i}$ and $z=x_{n,i+1}$. Then
$\deg_y\phhi_{pn+i}=\deg_z\phhi_{pn+i}=n-p-1$ and the coefficients at $y^{n-p-1}$ and $z^{n-p-1}$
in $\phhi_{pn+i}$ are equal to $\phhi_{(p-1)(n-1)+i}$ for the size $n-1$.

Finally, for the case of $\phhi_{pn}$, $1\le p\le n-2$, one takes $y=x_{n-1,n}$ and $z=y_{n1}$. 
Then $\deg_y\phhi_{pn}=n-p$ and $\deg_z\phhi_{pn}=n-p-1$,
the coefficient at $y^{n-p}$ in $\phhi_{pn}$ equals $\phhi_{(p-1)(n-1)+1}$ for the size $n-1$, while the coefficient 
at $z^{n-p-1}$ in $\phhi_{pn}$ equals the product of $\phhi_{p(n-1)}$ for the size $n-1$ by the determinant of an 
$(n-1)\times(n-1)$ matrix of independent variables. Further details are left to the interested reader.

Irreducibility of the remaining functions in the family $\FFF_n$ is discussed in~\cite[Section~6.3]{GSVDouble}.
\end{proof}

\begin{proof}[Proof of Lemma~\ref{fknirr}]
Irreducibility of the functions in $\FFF_{2n}$ is trivial. For $k>2$ we have to deal separately with functions $\tilde c_t$
and $\tilde\phhi_t$. 

To prove irreducibility of $\tilde c_t$, $1\le t\le k-1$, note that each such function is linear in all $a_{ij}$, 
$1\le i\le k+1$, $1\le j\le n$. Assume that $\tilde c_t=P_1P_2$ and that $P_1$ is linear in $a_{k+1,1}$
(and hence $P_2$ does not depend on $a_{k+1,1}$). Moreover, $P_1$ depends linearly in all nonzero entries in the 
first row of $X$ and in the first column
of $Y$, whereas $P_2$ does not depend on any of these entries.
Note that for any $a_{ij}$ as above there exists a staircase sequence 
$a_{k+1,1}=a_{i_0j_0},a_{i_1j_1},\dots,a_{i_lj_l},\dots,a_{i_rj_r}=a_{ij}$ such that $1\le i_l\le k+1$, $1\le j_l\le n$ for
$1<l<r$ and
every consecutive pair $(a_{i_{l-1}j_{l-1}},a_{i_lj_l})$ alternately lies in the same row or in the same column of the
matrix $X+Y$. Moving along this sequence and applying the same reasoning as above, we consecutively get that $P_1$ is 
linear in $a_{i_1j_1}, a_{i_2j_2},\dots,a_{ij}$,
and hence $P_2$ does not depend on $a_{ij}$, which means that $P_2$ is a constant.

Irreducibility of $\tilde \phhi_t$ is proved similarly to the proof of Lemma \ref{coreirr}. Below we sketch the proof for
$\tilde\phhi_1$; cases $t>1$ are treated in a similar way.

Let $\tilde\phhi_1$ be of type $(n,k)$. Take $x=a_{k+1,n}$, then $\deg_x\tilde\phhi_1=k-2$. The coefficient $c^x$ at $x^{k-2}$
in $\tilde\phhi_1$ equals $\psi_1^xa_{12}\cdots a_{1,k-1}D^x$ where $D^x=\det Y_{[k,n]}^{[1,n-k+1]}$. Consider 
$\psi_1^x$ as a polynomial of degree $(k-2)(n-k-2)$ in variables $a_{1,k+1}, \dots, a_{1,n-1}$. The constant term of
this polynomial is $\tilde\phhi_1$ of type $(n-1,k-1)$ for a shifted set of variables, and hence is irreducible by the induction hypothesis. Consequently, $\psi_1^x$ is irreducible since it is homogeneous as a polynomial in all variables 
$a_{ij}$. 

Assume that $\tilde\phhi_1=P'P''$. It follows from above that $P'=x^d\psi_1^xR'+o(x^d)$ and $P''=x^{k-2-d}R''+
o(x^{k-2-d})$ with $R'R''=a_{12}\cdots a_{1,k-1}D^x$. Consequently, $\deg P'\ge (k-2)(n-2)+d$ and $\deg P''\le n+k-d-3$.

Take $y=a_{k2}$, then $\deg_y\tilde\phhi_1=k-1$ and $\deg_y P'\ge\deg_y \psi_1^x=k-2$, and hence $\deg_y P''\le 1$.
Further, the coefficient $c^y$ at $y^{k-1}$ in $\tilde\phhi_1$ equals $\psi_1^ya_{13}\cdots a_{1k}D^y$ where
$D^y=\det Y_{[k+1,n]}^{[2,n-k+1]}$ and $\psi_1^y$ has the same structure as $\psi_1^x$. If $\deg_y P''=0$
then $P''$ is a factor of $c^y$, which is impossible. Consequently, $\deg_y P''=1$, and hence either 
\[
P'=y^{k-2}S'+o(y^{k-2}) \quad\text{and}\quad P''=y\psi_1^y S''+o(y)\quad\text{with}\quad S'S''= a_{13}\cdots a_{1k}D^y
\]
or 
\[
P'=y^{k-2}\psi_1^yS'+o(y^{k-2}) \quad\text{and}\quad P''=yS''+o(y)\quad\text{with}\quad S'S''= a_{13}\cdots a_{1k}D^y.
\]

In the first of the above two cases we get $\deg P''\ge (k-2)(n-2)+1$, which is strictly greater than $n+k-d-3$ for $k>3$. 
For $k=3$ either $\deg_x P'=0$ or $\deg_x P''=0$, which is impossible for the same reason as $\deg_yP''=0$. Consequently,
the second case holds true, and hence $\deg_xP'=k-3$ and $\deg_x P'=1$.Taking into account the reasoning above, we can write
$P''=xR''+yS''+T''$, where $T''$ does not depend on $x$ and $y$.

Consider now the coefficient $c$ at $x^{k-3}y^{k-1}$ in $\tilde\phhi_1$. On the one hand, $c$ is equal to $c^{xy}R'S''$, where
$c^{xy}$ is the coefficient at $y^{k-2}$ in $\psi_1^x$. The latter equals $\psi_1^{xy}a_{13}\cdots a_{1,k-1} D^{xy}$,
where $D^{xy}=\det Y^{[2,n-k+1]}_{[k,n-1]}$. On the other hand, $c$ is equal to $c^{yx}a_{13}\cdots a_{1k}D^y$, where $c^{yx}$
is the coefficient at $x^{k-3}$ in $\psi_1^y$. The latter equals $\psi_1^{yx}a_{13}\cdots a_{1,k-1} D^{yx}$,
where $D^{yx}=\det Y^{[2,n-k+2]}_{[k,n]}$. It is easy to see that $\psi_1^{xy}=\psi_1^{yx}$, hence we arrive at
$D^{xy}R'S''=D^{yx}a_{13}\cdots a_{1k} D^y$, which is clearly impossible.
\end{proof}

\subsection{Coprimality results}\label{varcop}

\begin{proof}[Proof of Lemma~\ref{gencop}] 
(i) Let $A$ be a matrix with distinct non-zero eigenvalues: $A=C^{-1}\diag(\lambda_1,\dots,\lambda_k)C$ with $\det C\ne0$.
We follow the proof of \cite[Proposition 8.1]{GSVDouble} and write
\[
\det K(A; e_1)=\Van(\lambda_1,\dots,\lambda_k)\prod_{i=1}^k c_{i1}
\]
and
\[
\det K^*(A; e_1, A^{-1}(e_2+\gamma e_1))=\Van(\lambda_1,\dots,\lambda_k)\prod_{i=1}^k w_{i}
\]
where $\Van$ is the Vandermond determinant and
\[
w_i=\sum_{j\ne i}\pm(c_{j2}+\gamma c_{j1})\lambda_j^{-1}\Van(\lambda_1,\dots,\hat\lambda_i,\dots,\hat\lambda_j,\dots,\lambda_k)
\prod_{m\ne i,j}c_{m1}.
\]
Pick $c_{11}=0$, then for any choice of $c_{i1}\ne 0$, $2\le i\le k$, and $c_{12}\ne 0$ one has $\det K(A; e_1)=0$ and
\[
w_i=\pm c_{12}\lambda_1^{-1}\Van(\lambda_2,\dots,\hat\lambda_i,\dots,\lambda_k)\prod_{m\ne 1,i}c_{m1}\ne 0,\qquad 2\le i\le k.
\]
Next, pick $\lambda_i=t^i$. Then for $t$ big enough the expression 
$\lambda_j^{-1}\Van(\lambda_2,\dots,\hat\lambda_j,\dots,\lambda_k)$ grows as $t^{\varkappa-j^2+2j}$ where $\varkappa$ depends only on $k$. Consequently, the leading term in the expression 
\[
w_1=\sum_{j=2}^k\pm(c_{j2}+\gamma c_{j1})\lambda_j^{-1}\Van(\lambda_2,\dots,\hat\lambda_j,\dots,\lambda_k)\prod_{m\ne 1,j}c_{m1}
\]
is obtained for $j=2$, and it suffices to pick
$c_{22}$ such that $c_{22}+\gamma c_{21}\ne 0$ to guarantee $w_1\ne 0$. The rest 
of $c_{ij}$ can be picked arbitrarily to 
satisfy condition $\det C\ne 0$, which yields  
$\det K^*(A; e_1, A^{-1}(e_2+\gamma e_1))\ne0$.

(ii) We have to refine the choice of $c_{ij}$ made in the proof of part (i). Note that an arbitrary principal leading 
minor of $A$ can be written as
\[
\det A_I^I=\sum_K\pm\det (C^{-1})_I^K\prod_{i\in K}\lambda_i\det C_K^I
=\frac1{\det C}\sum_K\pm\det C_{K^c}^{I^c}\prod_{i\in K}\lambda_i\det C_K^I,
\]
where $I=\{1,2,\dots,|I|\}$ and $|K|=|I|$. Our choice of $\lambda_i$ guarantees that for $t$ big enough the leading term in the above expression is
obtained when $K=\{k-|I|+1, k-|I|+2,\dots,k\}$. Consequently, condition $\det A_I^I\ne 0$ is guaranteed by $\det C_K^I\ne 0$
and $\det C_{K^c}^{I^c}\ne 0$. Clearly, these conditons can be satisfied via a suitable choice of the entries $c_{ij}$ distinct
from $c_{i1}$, $c_{12}$, and $c_{22}$. 
\end{proof}

\begin{proof}[Proof of Lemma~\ref{fncop}]
The claim for $f=\phhi_1$ follows from Lemma~\ref{gencop}(i). 
Indeed, fix an invertible $Y$ such that $\det\bar Y\ne 0$ and define $\gamma$
via~\eqref{gamma}. Next, pick $A$ that satisfies the conditions in Lemma~\ref{gencop}(i)
and put $X=A^{-1}Y$. Consequently, $\phhi_1(X,Y)=\det K(A;e_1)(\det X)^{n-1}=0$, while
\[
\phhi_1^*(X,Y)=\pm\det K^*(A;e_1,A^{-1}\bar v)(\det X)^{(n-1)(n-2)}(\det \bar Y)^n \ne 0
\]
 via~\eqref{phi_star}, and the claim follows. 

For $f=\phhi_i$, $2\le i\le n$, the claim is trivial, since in this range $\deg f^*<\deg f$. In the case $f=\phhi_i$, 
$n+1\le i\le N-n$, it follows from the explanations in the proof of Theorem~\ref{structure} that 
$\phhi_i^*=\phhi_{i+n}\phhi^0_{i-n}-\phhi_{i-n}\phhi^0_{i+n}$, where $\phhi_t^0$ is the
minor of $\Phi$ obtained by replacing the first column of $\phhi_t$ by the immediately preceding column.
Consider the specialization that sets to zero the entry $z$  ($y_{st}$  or $x_{st}$) that occupies position 
$(i,i)$ in $\Phi$ 
and all the entries of the matrices X and Y that lie in the same columns of $\Phi$ below $z$. 
It is easy to see that this specialization  implies vanishing of $\phhi_i$ and 
$\phhi_{i+n}$, since both minors acquire a zero column. However, the same specialization for $\phhi_{i-n}$ and
$\phhi_{i+n}^0$ yields nontrivial polynomials since the coefficient at $\bar z^{n-p}$ in the first one and
at $\bar z^{n-p-3}$ in the second one are nontrivial, where $\bar z$ lies immediately above $z$ in $\Phi$
and $p=\lfloor(i-1)/n\rfloor$. Consequently, $\tilde\phhi_i^*$ is not divisible by $\tilde\phhi_i$.
The cases $f=g_{ii}$, $2\le i\le n$, and  $f=h_{ii}$, $3\le i\le n$, are treated via the same specialization.
For $f=h_{22}$ the specialization is given by $y_{nj}=0$ for $2\le j\le n$.
\end{proof}

\begin{proof}[Proof of Lemma~\ref{fkncop}]
To prove the coprimality of $\tilde\phhi_t$ and $\tilde\phhi_t^*$ it suffices to check that the latter is not divisible by the former. For $t=1$ we need the following statement.

\begin{proposition}
\label{Lemma_3}
The image of the map $(X,Y)\mapsto U$ defined by \eqref{bandU} contains an arbitrary $k\times k$ matrix with nonzero trailing principal minors.
\end{proposition}

\begin{proof} In what follows $B$, $B_1$, $\bar B_1$ are $k\times k$ invertible upper triangular and $N$, $\bar N$, $N_1$
are $k\times k$ unipotent lower triangular. 
It suffices to show that for any $B$ and $N$ as above 
there exist $n\times n$ matrices $X$, $Y$ of the form~\eqref{XYband} such that $U$ defined by~\eqref{bandU} is given by $U= B N$.

Let $Y_0$ of the form described in \eqref{XYband} be totally nonnegative with  all combinatorially nontrivial minors nonzero and set $Y_1= J Y_0 J$, where $J=\diag \left ( (-1)^i)_{i=0}^{n-1} \right)$. Then $M=\left (Y_1^{-1} \right )_{[n-k+1,n]}^{[1,k]}$  is totally nonnegative and invertible since 
$\det M=  \det Y_1^{-1} \det {Y_1}_{[k+1,n]}^{[1,n-k]}$. Thus there exist $N_1$ and $B_1$ as above such that $M= N_1B_1$.

It is not hard to see that there exists an invertible positive diagonal matrix $D$ such that $NDB_1^{-1}=\bar B_1^{-1}\bar N$ for $\bar B_1$ and $\bar N$ as above. Let $Y_2$ be obtained via multiplying $Y_1$
on the left by an appropriate diagonal matrix so that
that $\left (Y_2^{-1} \right )_{[n-k+1,n]}^{[1,k]}=MD^{-1}$. Now, let 
$Y = Y_2 \begin{bmatrix} \one_{n-k} & 0\\ 0 &  N_1\bar N^{-1}\end{bmatrix}$ and 
$ X = \begin{bmatrix} 0 & B\bar B_1^{-1}\\ 0 & 0 \end{bmatrix}$. Then $X$, $Y$ are of the required form, 
$\left (Y^{-1} \right )_{[n-k+1,n]}^{[1,k]} = \bar N N_1^{-1}MD^{-1}=\bar N B_1D^{-1}$  and~\eqref{bandU} gives 
$U=B{\bar B_1}^{-1} \bar N B_1D^{-1}= BNDB_1^{-1}B_1D^{-1}= B N$.
\end{proof}

To complete the proof for $t=1$ we invoke Lemma~\ref{gencop}(ii) which guarantees that one can choose $A$ in Lemma~\ref{gencop}(i)
in such a way that the trailing principal minors of $A^{-1}$
are nonzero, and proceed as in the proof of Lemma~\ref{fncop}.

For $2\le t\le n$ the claim is trivial, since in this case $\deg\tilde\phhi_t^*<\deg\tilde\phhi_t$.

Let now $n< t < (k-2)(n-1)$. Then it follows from the explanations above that 
$\tilde\phhi_t^*=\tilde\phhi_{t+n}\tilde\phhi^0_{t-n}-\tilde\phhi_{t-n}\tilde\phhi^0_{t+n}$, same as in the previous section.
Find the unique pair $(i,j)$, $1\le i\le n-1$, satisfying $t=(k-j)(n-1)+i$ 
and consider the specialization $a_{ij}=a_{i-1,j+1}=\cdots=a_{i-p,j+p}=\cdots=a_{1,j+i-1}=0$, where the index $j+p$ is understood $\mod n$ 
and its value $j+p=1$ is excluded.
It is easy to see that this specialization  implies vanishing of $\tilde\phhi_t$ and 
$\tilde\phhi_{t+n}$, since both minors acquire a zero column. However, the same specialization for $\tilde\phhi_{t-n}$ and
$\tilde\phhi_{t+n}^0$ yields nontrivial polynomials since the coefficient at $a_{i+1,j-1}^j$ in the first one and
at $a_{i+1,j-1}^{j-3}$ in the second one are nontrivial. Consequently, $\tilde\phhi_t^*$ is not divisible by $\tilde\phhi_t$.

Finally, let $n\ge (k-2)(n-1)$. Then it follows from the explanations above that $\tilde\phhi_t^*=\tilde\phhi^0_{t-n}$, and the same 
specialization as above proves that $\tilde\phhi_t^*$ is not divisible by $\tilde\phhi_t$.
\end{proof}

\section*{Acknowledgments}

Our research was supported in part by the NSF research grant DMS \#1702054 (M.~G.), NSF research grant DMS \#1702115 (M.~S.), and ISF grant \#1144/16 (A.~V.). While working on this project, we benefited from support of the following institutions 
and programs: Centre International de Rencontres Math\'ematiques, Lumini (M.~G. and M.~S., Spring 2018), LAREMA, Universite
d'Angers (M.~G., M.~S., A.~V., Summer 2018), Mathematical Institute of the University of Heidelberg (M.~G., Summer 2018), University of Notre Dame Jerusalem Global Gateway and the University of Haifa (M.~G., M.~S., A.~V., Fall 2018), Michigan State University (A.~V., Spring 2019), Research Institute for Mathematical Sciences, Kyoto (M.~G., M.~S., A.~V., Spring 2019), Research in Pairs Program at the Mathematisches Forschungsinstitut Oberwolfach (M.~S. and A.~V., Summer 2019), Istituto Nazionale di Alta Matematica Francesco Severi and the Sapienza University of Rome (A.~V., Fall 2019) and the Mathematical Science Research Institute, Berkeley (M.~S., Fall 2019).  We are grateful to all these institutions for their hospitality and outstanding working conditions they provided. Special thanks are due to Bernard Leclerc who pointed to us references to papers \cite{Gl} and~\cite{LV}.


\begin{thebibliography}{00}

\bibitem{CAIII}  A.~Berenstein, S.~Fomin, and A.~Zelevinsky,
\textit{Cluster algebras. III. Upper bounds and double Bruhat cells}.
Duke Math. J. \textbf{126} (2005), 1--52.


\bibitem{CS} L.~Chekhov and M.~Shapiro, {\it Teichm\"uller spaces of Riemann surfaces with orbifold points of arbitrary order and cluster variables}, IMRN (2014), no.~10, 2746--2772.

\bibitem{FZ1} S.~Fomin, A.~Zelevinsky, {\t Double Bruhat cells and total positivity}, J. Amer. Math. Soc. {\bf 12} (1999), no. 2, 335--380.

\bibitem{FZ2} S.~Fomin, A.~Zelevinsky, {\it Cluster algebras I:
Foundations}, J. Amer. Math. Soc. {\bf 15} (2002), 497--529.

\bibitem{GSV1} M.~Gekhtman, M.~Shapiro, A.~Vainshtein, {\it Cluster algebras
and Poisson geometry},  Mosc.
Math. J. {\bf 3} (2003), 899--934.

\bibitem{GSVMMJ}  M.~Gekhtman, M.~Shapiro, and A.~Vainshtein,  
\textit{Cluster structures on simple complex Lie groups and Belavin--Drinfeld classification},
Mosc. Math. J. \textbf{12} (2012), 293--312.

\bibitem{GSVCR}  M.~Gekhtman, M.~Shapiro, and A.~Vainshtein,  {\it Generalized cluster structure on the Drinfeld double of
$GL_n$}, C.~R.~Math.~Acad.~Sci.~Paris {\bf 354} (2016), 345-–349. 
 
\bibitem{GSVMem}  M.~Gekhtman, M.~Shapiro, and A.~Vainshtein,  {\it Exotic cluster structures on $SL_n$: the Cremmer--Gervais case},  Mem. Amer. Math. Soc. {\bf 246} (2017), no. 1165.


\bibitem{GSVDouble}  M.~Gekhtman, M.~Shapiro, and A.~Vainshtein,  {\it Drinfeld double of $GL_n$ and generalized cluster structures}. Proc. Lond. Math. Soc. {\bf 116} (2018),  429--484.

\bibitem{GSVPleth}  M.~Gekhtman, M.~Shapiro, and A.~Vainshtein,  {\it Plethora of cluster structures on $GL_n$}, arXiv:1902.02902.  

\bibitem{Gl} A.-S.~Gleitz, {\it Representations of $U_q(Lsl_2)$ at roots of unity and generalized cluster algebras}, European J. Combin. {\bf 57} (2016), 94--108. 

\bibitem{Anton} A.~Izosimov, {\it Pentagram maps and refactorization in Poisson-Lie groups}, arXiv:1803.00726.

\bibitem{LV} D.~Labardini-Fragoso and D.~Velasco, {\it On a family of Caldero--Chapoton algebras that have the Laurent 
phenomenon}, J. Algebra {\bf 520} (2019), 90--135.

\bibitem{LFZ}  D.~Labardini-Fragoso and A.~Zelevinsky, {\it Strongly primitive species with potentials I: mutations}. Bol. Soc. Mat. Mex. {\bf 22} (2016), 47-–115.


\bibitem{Leclerc} B.~Leclerc, {\it Cluster structures on strata of flag varieties},
Adv. Math. {\bf  300} (2016), 190--228. 

\bibitem{vMM} P.~van Moerbeke and D.~Mumford,{\it The spectrum of difference operators and algebraic curves},
Acta Math. {\bf 143} (1979),  93--154. 

\bibitem{MOST}
S.~Morier-Genoud, V.~Ovsienko, R.~Schwartz, and S.~ Tabachnikov, {\it Linear difference equations, frieze patterns, and the combinatorial Gale transform}, 
Forum Math. Sigma {\bf 2} (2014), e22, 45 pp. 

\bibitem{Scott} J. S. ~Scott,  {\it Grassmannians and cluster algebras}, Proc. Lond. Math. Soc. {\bf 92} (2006), 345--380.

\bibitem{YZ} S.-W.~Yang and A.~Zelevinsky, 
{\it Cluster algebras of finite type via Coxeter elements and principal minors},
Transform. Groups {\bf 13} (2008), 855--895. 


\end{thebibliography}
\end{document}